\documentclass[12pt]{amsart}
\usepackage[margin=1in]{geometry}
\usepackage{amsmath}
\usepackage{amsfonts}
\usepackage[scr=boondox, cal=esstix]{mathalpha}
\usepackage{amsthm}
\usepackage{mathrsfs}
\usepackage{amssymb}
\usepackage{tikz}
\usepackage{comment}
\usetikzlibrary{calc}
\usetikzlibrary{arrows.meta}
\usepackage{tikz-cd}
\usepackage{graphicx}
\usepackage{mathtools}
\usepackage{setspace}
\usepackage{blkarray}
\usepackage{stmaryrd}
\usepackage[dvipsnames]{xcolor}
\usepackage{nicematrix}
\usepackage{ytableau}

\usepackage{blkarray}

\usepackage{cite}
\usepackage{hyperref}
\hypersetup{hidelinks}

\usepackage{amscd}

\renewcommand{\span}{\mathsf{span}}
\renewcommand{\d}{\mathbf{d}}

\newcommand{\bs}{\boldsymbol}

\DeclareMathOperator{\rep}{\mathbf{rep}}

\DeclareMathOperator{\Mat}{Mat}

\DeclareMathOperator{\wt}{wt}

\DeclareMathOperator{\rank}{\operatorname{rank}}
\DeclareMathOperator{\Std}{\operatorname{Adm}}
\DeclareMathOperator{\tOmega}{\widetilde{\Omega}}
\DeclareMathOperator{\tabe}{{\sf e}}
\DeclareMathOperator{\tabf}{{\sf f}}
\DeclareMathOperator{\worde}{{\mathcal e}}
\DeclareMathOperator{\wordf}{{\mathcal f}}
\DeclareMathOperator{\mate}{\sf E}
\DeclareMathOperator{\matf}{\sf F}

\DeclareMathOperator{\ins}{\operatorname{ins}}

\newcommand{\C}{\mathbb C}

\newcommand{\quivermatrix}[3]{
\scalebox{0.8}{
\begin{tikzpicture}[baseline=-0.5ex]
  \draw[line width=.2pt, step=.25cm] (0,0) grid (#1,#2);
  #3
\end{tikzpicture}
}
}

\newcommand{\cell}{.28}

\newcommand{\tile}[5]{%
\begin{scope}[shift={(#1,#2)}]
  \draw[line width =.1pt, step=\cell] (0,0) grid ({#3*\cell},{#4*\cell});
  #5
\end{scope}
}

\newcommand{\phantomtile}[5]{%
\begin{scope}[shift={(#1,#2)}]
  \fill[gray!15] (0,0) rectangle ({#3*\cell},{#4*\cell});
  \draw[gray!70, dashed, step=\cell] (0,0) grid ({#3*\cell},{#4*\cell});
  #5
\end{scope}
}
\newcommand{\dotcell}[2]{%
  \fill ({(#1+.5)*\cell},{(2.5-#2)*\cell}) circle (2pt);
}

\newcommand{\numcell}[3]{%
  \node at ({(#1+.5)*\cell},{(2.5-#2)*\cell}) {\scriptsize $#3$};
}

\newtheorem{theorem}{Theorem}[section]
\newtheorem{lemma}[theorem]{Lemma}   
\newtheorem{corollary}[theorem]{Corollary}

\newtheorem{proposition}[theorem]{Proposition}

\newtheorem{claim}[theorem]{Claim}
\newtheorem*{theorem*}{Theorem}

\theoremstyle{definition} 
\newtheorem{definition}[theorem]{Definition}
\newtheorem{example}[theorem]{Example}

\theoremstyle{remark}
\newtheorem{remark}[theorem]{Remark}

\usepackage{shadowtext}
\usepackage{xcolor}

\shadowcolor{gray!50}              

\newcommand{\barindex}[1]{\overline{#1}}

\usepackage{nicematrix}
\usepackage{array}

\begingroup
\catcode`\&=13
\gdef\pampmatrix{%
  \begingroup
  \let&=\amsamp
  \begin{pmatrix}%
}
\gdef\endpampmatrix{\end{pmatrix}\endgroup}
\endgroup

\usepackage{nicematrix}

\begin{document}
\pagestyle{plain}

\title{A combinatorial rule for\\ $GL$-multiplicities of $A_n$-quiver loci}
  
\author{Ian Cavey, Andrew Hardt, and Alexander Yong}

\address{Department of Mathematics, U.~Illinois Urbana-Champaign, Urbana, IL 61801, USA}
\email{cavey@illinois.edu, ahardt@illinois.edu, ayong@illinois.edu}
\date{August 6, 2026}

\begin{abstract}
We give the first positive combinatorial rule for the multiplicities of irreducible $GL$-representations
in the coordinate rings of type $A$ quiver orbit closures, valid for every orientation of the quiver.
Previously, for the special case of varieties of complexes, work of De Concini--Strickland in the early 1980s gave an implicit description of these multiplicities. The combinatorial objects in our rule carry a crystal structure whose highest-weight elements compute these multiplicities.
\end{abstract}

\maketitle
\section{Introduction}

\subsection{Overview}

Quiver representations are collections of linear maps arranged along the edges of a directed graph. 
For quivers of type $A$, their orbit closures are algebraic varieties defined by systems of rank conditions, and may be viewed as generalizations of classical degeneracy loci. This places them in a tradition going back to Giambelli \cite{FultonPragacz}, Thom--Porteous \cite{Porteous}, and Kempf--Laksov \cite{KempfLaksov}, and continuing through Fulton \cite{Fulton-duke}, Buch--Fulton \cite{BuchFulton}, and others.

For type $A$ quivers, the equations of these orbit closures are understood through work of Zelevinsky \cite{Zelevinsky} and Kinser--Rajchgot \cite{KinserRajchgot}. 
Formulas for their Hilbert series were also known;
see Knutson--Miller--Shimozono \cite{KMS}, Buch \cite{BuchJAMS}, Miller \cite{MillerDuke}, 
Kinser--Knutson--Rajchgot \cite{KKR}, and the references therein. 
Missing was a positive combinatorial rule for the multiplicities in the irreducible decomposition of their coordinate
rings under the base change group.

We give  such a rule uniformly for all type $A$ quiver orbit closures and every orientation of the quiver.
Our rule arises from a positive combinatorial model for the character of the coordinate ring, equipped with a crystal structure whose highest-weight elements compute these multiplicities.
It may be seen as a generalized Littlewood--Richardson rule.

\subsection{Main result}\label{sec:mainresult}

An $A_n$ \emph{quiver} $Q$ is a directed path, and a \emph{quiver representation} assigns a finite-dimensional complex vector space to each vertex and a linear map to each arrow. Fixing a dimension vector $\mathbf d=(d_1,d_2,\ldots,d_n)$, the space 
$\rep_Q(\d)$ of all such representations is an affine space carrying a natural action of the \emph{base change group} 
$GL=\prod_{i=1}^n GL_{d_i}$. The orbits  $\Omega^{\circ}$ of this action correspond to isomorphism classes of representations, 
and their closures $\Omega$ are the \emph{quiver loci}. The coordinate ring $\C[\Omega]$ is a $GL$-representation. We determine, for each dominant weight ${\bf \Lambda}$,  the multiplicity of the irreducible $GL$-module $V_{\bf \Lambda}$ in $\C[\Omega]$.

After choosing bases, $\Omega^{\circ}$ contains a representation whose arrow maps $L_i$ are transposed
partial permutation matrices, of dimension $d_i\times d_{i+1}$ if  the arrow is $i\rightarrow i+1$, and
$d_{i+1}\times d_{i}$ otherwise. A tuple of such partial permutation matrices determines $\Omega^\circ$, and hence $\Omega$.
For example, 
\[{\mathbb C}^3 
\xrightarrow{
\quivermatrix{.75}{.75}{
  \fill (.125,.625) circle (2pt);
  \fill (.375,.375) circle (2pt);
}}
 {\mathbb C}^3 \xrightarrow{
 \quivermatrix{.75}{.75}{
  \fill (.375,.375) circle (2pt);
}
 } {\mathbb C}^3
\xleftarrow{
\quivermatrix{.75}{.5}{
  \fill (.125,.375) circle (2pt);
}
} {\mathbb C}^2.\]

Using an idea of \cite{KinserRajchgot}, replace $Q$ with a  bipartite quiver $\widetilde Q$ by iterating the local rules:

\[{\mathbb C}^{d_{i-1}} \xrightarrow{L_{i-1}} {\mathbb C}^{d_{i}}
 \xrightarrow{L_i}
 {\mathbb C}^{d_{i+1}} \leadsto
 {\mathbb C}^{d_{i-1}} \xrightarrow{L_{i-1}} {\mathbb C}^{d_{i}}
 {\color{red}{\xleftarrow{I_{d_i}}} {\mathbb C}^{d_i}} \xrightarrow{L_{i}} 
 {\mathbb C}^{d_{i+1}}, 
 \]
 \[{\mathbb C}^{d_{i-1}} \xleftarrow{L_{i-1}} {\mathbb C}^{d_{i}}
 \xleftarrow{L_i}
 {\mathbb C}^{d_{i+1}} \leadsto
 {\mathbb C}^{d_{i-1}} \xleftarrow{L_{i-1}} {\mathbb C}^{d_{i}}
 {\color{red}{\xrightarrow{I_{d_i}}} {\mathbb C}^{d_i}} \xleftarrow{L_{i}} 
 {\mathbb C}^{d_{i+1}}, 
 \]
where $I_{d_i}$ is the $d_i\times d_i$ identity matrix. Draw $\widetilde Q$ in a zig-zag 
so  each $\leftarrow$ now points down. Place each $L_i$ or $I_{d_i}$ 
below or to the left of its associated arrow in the zig-zag. In our example, $\widetilde Q$ and the associated
\emph{$Q$-shape} are shown below:
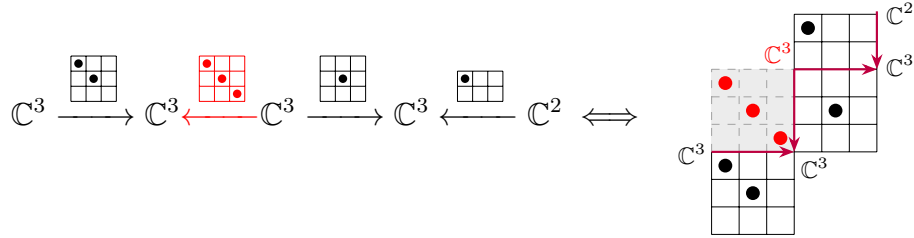
\[
{\mathbb C}^3 
\xrightarrow{
\quivermatrix{.75}{.75}{
  \fill (.125,.625) circle (2pt);
  \fill (.375,.375) circle (2pt);
}}
{\mathbb C}^3 {\color{red}{\xleftarrow{ 
\quivermatrix{.75}{.75}{
 \fill (.125,.625) circle (2pt);
  \fill (.375,.375) circle (2pt);
  \fill (.625,.125) circle (2pt);
 }}}}
 {\mathbb C}^3 \xrightarrow{
 \quivermatrix{.75}{.75}{
  \fill (.375,.375) circle (2pt);
}
 } {\mathbb C}^3
\xleftarrow{
\quivermatrix{.75}{.5}{
  \fill (.125,.375) circle (2pt);
}
} {\mathbb C}^2
\iff
\vcenter{\hbox{
\begin{tikzpicture}[scale=1.3, baseline=-.5ex]

\tile{0}{0}{3}{3}{
  \dotcell{0}{0}
  \dotcell{1}{1}
}

{\color{red}{\phantomtile{0}{.84}{3}{3}{
  \dotcell{0}{0}
  \dotcell{1}{1}
  \dotcell{2}{2}
}}}

\tile{.84}{.84}{3}{3}{
  \dotcell{1}{1}
}

\tile{.84}{1.68}{3}{2}{
  \dotcell{0}{1}
}

\draw[purple, -{Stealth[length=2mm,width=1.5mm]}, thick] (0,.84) -- (.84,.84);
\draw[purple, -{Stealth[length=2mm,width=1.5mm]}, thick] (.84,1.68) -- (.84,.84);
\draw[purple, -{Stealth[length=2mm,width=1.5mm]}, thick] (.84,1.68) -- (1.68,1.68);
\draw[purple, -{Stealth[length=2mm,width=1.5mm]}, thick] (1.68,2.27) -- (1.68,1.68);

\node at (-.2,.82) {\scriptsize ${\mathbb C}^3$};
\node at (1.05,.68) {\scriptsize ${\mathbb C}^3$};
\node at (0.68,1.85) {\scriptsize {\color{red}{${\mathbb C}^3$}}};
\node at (1.92,1.72) {\scriptsize ${\mathbb C}^3$};
\node at (1.92,2.25) {\scriptsize ${\mathbb C}^2$};
\end{tikzpicture}}}
\]
Label the regions ${\mathcal R}_1,{\mathcal R}_2,\ldots, {\mathcal R}_{m}$ from southwest to northeast, where ${\mathcal R}_i$
is filled by the matrix between vertices $i,i+1$ of ${\widetilde Q}$.
For $1\leq i\leq j\leq m$,  $r_{[i,j]}$ is the rank of this filling of
${\mathcal R}_{[i,j]}={\mathcal R}_i \cup {\mathcal R}_{i+1} \cup \cdots \cup {\mathcal R}_j$. For instance, in the example, $r_{[1]}=2, r_{[1,3]}=4, r_{[3,4]}=2$.
A region filled by a newly introduced identity matrix (depicted in red) is a
\emph{phantom region}.

Our rule counts certain fillings ${\bf M}$ of the $Q$-shape by nonnegative integers. An \emph{antidiagonal} of ${\bf M}$ consists of positions  $(i_1,j_1),\dots,(i_k,j_k)$ with $i_1<\cdots<i_k$ and $j_1>\cdots>j_k$ 
such that each position either contains a positive entry of ${\bf M}$ or lies in a 
phantom region. It is \emph{phantom} if it uses at least one position in a phantom region. Define
${\bf M}$ to be $\Omega$-\emph{admissible} if, for every $1\leq i\leq j\leq m$, the portion in ${\mathcal R}_{[i,j]}$ contains no (phantom) 
antidiagonal of length $r_{[i,j]}+1$, and if it is identically zero on every phantom region. Let $\Std(\Omega)$ be the set of 
such ${\bf M}=(M_1,\ldots,M_{n-1})$ with phantom regions omitted. E.g., 
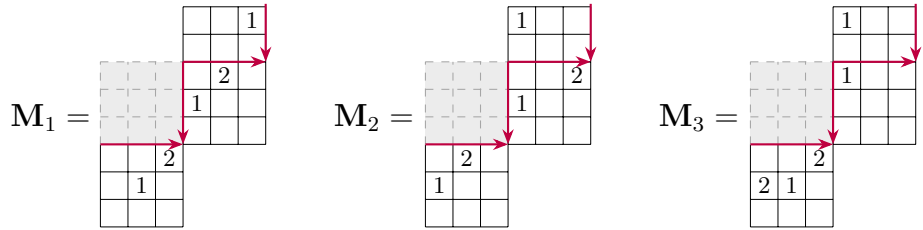
\[{\bf M}_1=\vcenter{\hbox{\begin{tikzpicture}[scale=1.3, baseline=-.5ex]

\tile{0}{0}{3}{3}{
  \numcell{2}{0}{2}
  \numcell{1}{1}{1}
}

{\phantomtile{0}{.84}{3}{3}{
}}

\tile{.84}{.84}{3}{3}{
  \numcell{0}{1}{1}
  \numcell{1}{0}{2}
}

\tile{.84}{1.68}{3}{2}{
  \numcell{2}{1}{1}
}

\draw[purple, -{Stealth[length=2mm,width=1.5mm]}, thick] (0,.84) -- (.84,.84);
\draw[purple, -{Stealth[length=2mm,width=1.5mm]}, thick] (.84,1.68) -- (.84,.84);
\draw[purple, -{Stealth[length=2mm,width=1.5mm]}, thick] (.84,1.68) -- (1.68,1.68);
\draw[purple, -{Stealth[length=2mm,width=1.5mm]}, thick] (1.68,2.27) -- (1.68,1.68);
\end{tikzpicture}}}
\qquad
{\bf M}_2=\vcenter{\hbox{\begin{tikzpicture}[scale=1.3, baseline=-.5ex]

\tile{0}{0}{3}{3}{
  \numcell{1}{0}{2}
  \numcell{0}{1}{1}
}

{\phantomtile{0}{.84}{3}{3}{
}}

\tile{.84}{.84}{3}{3}{
  \numcell{0}{1}{1}
  \numcell{2}{0}{2}
}

\tile{.84}{1.68}{3}{2}{
  \numcell{0}{1}{1}
}

\draw[purple, -{Stealth[length=2mm,width=1.5mm]}, thick] (0,.84) -- (.84,.84);
\draw[purple, -{Stealth[length=2mm,width=1.5mm]}, thick] (.84,1.68) -- (.84,.84);
\draw[purple, -{Stealth[length=2mm,width=1.5mm]}, thick] (.84,1.68) -- (1.68,1.68);
\draw[purple, -{Stealth[length=2mm,width=1.5mm]}, thick] (1.68,2.27) -- (1.68,1.68);
\end{tikzpicture}}}
\qquad
{\bf M}_3=\vcenter{\hbox{\begin{tikzpicture}[scale=1.3, baseline=-.5ex]

\tile{0}{0}{3}{3}{
  \numcell{2}{0}{2}
  \numcell{1}{1}{1}
  \numcell{0}{1}{2}
}

{\phantomtile{0}{.84}{3}{3}{
}}

\tile{.84}{.84}{3}{3}{
  \numcell{0}{0}{1}
}

\tile{.84}{1.68}{3}{2}{
  \numcell{0}{1}{1}
}

\draw[purple, -{Stealth[length=2mm,width=1.5mm]}, thick] (0,.84) -- (.84,.84);
\draw[purple, -{Stealth[length=2mm,width=1.5mm]}, thick] (.84,1.68) -- (.84,.84);
\draw[purple, -{Stealth[length=2mm,width=1.5mm]}, thick] (.84,1.68) -- (1.68,1.68);
\draw[purple, -{Stealth[length=2mm,width=1.5mm]}, thick] (1.68,2.27) -- (1.68,1.68);
\end{tikzpicture}}}
\]
Above, ${\bf M}_1\not\in \Std(\Omega)$ as it has an antidiagonal in ${\mathcal R}_{[3,4]}$ of length
$3>r_{[3,4]}=2$. Also, ${\bf M}_2\not\in \Std(\Omega)$  as it contains a phantom antidiagonal
in ${\mathcal R}_{[1,3]}$ of length $5>r_{[1,3]}=4$. Finally ${\bf M}_3\in \Std(\Omega)$  (e.g. the entries in ${\mathcal R}_1$
do not form an antidiagonal of length $3$).

The \emph{row word} ${\operatorname{row}}(M)$ of a nonnegative integer matrix $M$ is obtained by reading columns top down, left to right, recording $M_{rc}$ many $r$'s. Its \emph{column word} ${\operatorname{col}}(M)$
reads the rows in English order and records $M_{rc}$ many $c$'s. If $M$ has $d$ columns, 
its \emph{reverse complement word} ${\operatorname{col}}^*(M)$ is obtained by reversing ${\operatorname{col}}(M)$ and replacing each letter $c\in [d]$ by the decreasing word
 $d(d-1)\cdots(c+1)(c-1)\cdots 1$. For instance, if $M= 
 \vcenter{\hbox{\begin{tikzpicture}[scale=0.8, baseline=-.2ex]
\tile{0}{0}{3}{3}{
  \numcell{2}{0}{2}
  \numcell{1}{1}{1}
  \numcell{0}{1}{2}
}
\end{tikzpicture}}}$,
 $\operatorname{row}(M)=22211$,
 $\operatorname{col}(M)=33112$ and $\operatorname{col}^*(M)=31\, 32 \, 32\, 21\, 21$.
 
If $M$ is associated to edge $e$ of $Q$, its \emph{$i$-reading word} for  an incident vertex $i$,
denoted $\operatorname{read}_i(M)$, is $\operatorname{row}(M)$ if $i$ is a tail of $e$ and 
${\operatorname{col}}^*(M)$ otherwise. 
The \emph{reading list} $\operatorname{read}_Q({\bf M})$ is the sequence
$\operatorname{read}_Q({\bf M})=
\left(\operatorname{read}_i(M_{i-1})\operatorname{read}_i(M_i)\right)_{i=1}^n$,
where $M_0$ and $M_n$ are declared empty.

Apply Robinson-Schensted-Knuth (RSK) insertion to the $n$ words in 
$\operatorname{read}_Q({\bf M})$. If the resulting tableaux $(T_1,\ldots,T_n)$ are all 
\emph{supersemistandard} (only $j$'s appear in row $j$), ${\bf M}$ is \emph{highest-weight}.
Let $\lambda^{(i)}$ be the shape of $T_i$ and $\Delta_i=$ the sum of the entries of matrices where $i$ is the head of the associated edge; let $\operatorname{profile}({\bf M})=([\lambda^{(i)},\Delta_i])_{i=1}^n$.

Let $V_{[\lambda^{(i)},\Delta_i]}(d_i)$ be the $GL_{d_i}$-representation 
$V_{\lambda^{(i)}}(d_i)\otimes
\det^{-\Delta_i}$, where $V_{\lambda^{(i)}}(d_i)$ is the $GL_{d_i}$-\emph{Schur module} 
of highest weight $\lambda^{(i)}$, and 
 $\det$ denotes the determinant representation of $GL_{d_i}$. Each $V_{[\lambda^{(i)},\Delta_i]}(d_i)$ is an irreducible
$GL_{d_i}$-rational representation of highest weight $\lambda^{(i)}-\Delta_i\cdot(1,1,\ldots,1)$.
Accordingly, when indexing representations, we identify $[\lambda^{(i)},\Delta_i]$ with the $GL_{d_i}$-weight
$\lambda^{(i)}-\Delta_i\cdot(1,1,\ldots,1)$ and regard $\operatorname{profile}({\bf M})$ as the resulting $GL$-weight $\bf\Lambda$. 
The representation $V_{\bf\Lambda}$ is the exterior tensor product of these $V_{[\lambda^{(i)},\Delta_i]}(d_i)$; every irreducible rational $GL$-representation is of this form.

Our main theorem gives the irreducible $GL$-decomposition of the coordinate ring of $\Omega$:

\begin{theorem}\label{thm:newrule}
In the irreducible $GL$-decomposition of $\mathbb{C}[\Omega]$,
\[
{\mathbb C}[\Omega]\cong_{GL} \bigoplus_{\Lambda} V_{{\bf \Lambda}}^{\oplus m_{\Omega, \bf \Lambda}},
\]
we have $m_{\Omega,{\bf \Lambda}}=\#\{{\bf M}\in\Std(\Omega): {\bf M}\text{ is highest-weight},\ 
\operatorname{profile}({\bf M})={\bf \Lambda}\}$.
\end{theorem}

For example, one sees ${\bf M}=\left(\vcenter{\hbox{\!\!\!\!
\quivermatrix{.75}{.75}{
   \node at (.625,.625) {\scriptsize $2$};
  \node at (.375,.375) {\scriptsize $1$};
}\!\!, \!\!\!\!
\quivermatrix{.75}{.75}{
  \node at (.625,.625) {\scriptsize $1$};
} \!\!\!,\!\!\!\!
\quivermatrix{.75}{.5}{
  \node at (.625,.375) {\scriptsize $1$};
}}}\!\!\!
\right)\in \Std(\Omega)$ from its $Q$-shape:
\[
{\mathbb C}^3 
\xrightarrow{
\quivermatrix{.75}{.75}{
   \node at (.625,.625) {\scriptsize $2$};
  \node at (.375,.375) {\scriptsize $1$};
}}
 {\mathbb C}^3 \xrightarrow{
 \quivermatrix{.75}{.75}{
  \node at (.625,.625) {\scriptsize $1$};
}
 } {\mathbb C}^3
\xleftarrow{
\quivermatrix{.75}{.5}{
  \node at (.625,.375) {\scriptsize $1$};
}
} {\mathbb C}^2
\iff
\vcenter{\hbox{
\begin{tikzpicture}[scale=1.2, baseline=-.5ex]
\tile{0}{0}{3}{3}{
  \numcell{2}{0}{2}
  \numcell{1}{1}{1}
}

{\phantomtile{0}{.84}{3}{3}{
}}

\tile{.84}{.84}{3}{3}{
  \numcell{2}{0}{1}
}

\tile{.84}{1.68}{3}{2}{
  \numcell{2}{1}{1}
}

\draw[purple, -{Stealth[length=2mm,width=1.5mm]}, thick] (0,.84) -- (.84,.84);
\draw[purple, -{Stealth[length=2mm,width=1.5mm]}, thick] (.84,1.68) -- (.84,.84);
\draw[purple, -{Stealth[length=2mm,width=1.5mm]}, thick] (.84,1.68) -- (1.68,1.68);
\draw[purple, -{Stealth[length=2mm,width=1.5mm]}, thick] (1.68,2.27) -- (1.68,1.68);
\end{tikzpicture}}}.
\]
Now, $\operatorname{read}_Q({\bf M})= (211, 312121\,1, 21\, 21,1)$. 
RSK insertion gives $\left(\vcenter{\hbox{\ytableausetup{boxsize=.8em} \ytableaushort{11, 2},
\ytableaushort{1111,22,3}, \ytableaushort{11,22},\ytableaushort{1}}}\right)$
and so ${\bf M}$ is highest-weight. Since
$\ytableausetup{boxsize=.5em}\operatorname{profile}({\bf M})=
\left(\left[\vcenter{\hbox{$\ydiagram{2,1}$}},0\right], \left[\vcenter{\hbox{$\ydiagram{4,2,1}$}}, 3\right], 
\left[\vcenter{\hbox{$\ydiagram{2,2}$}}, 2\right], \left[\ydiagram{1},0\right] \right)$,
${\bf M}$ witnesses in ${\mathbb C}[\Omega]$ the appearance of
\[V_{\bf \Lambda}=V_{\ydiagram{2,1}}(3)\boxtimes \left(
V_{\ydiagram{4,2,1}}(3)\otimes \det\ \!\!\!^{-3}
\right)
\boxtimes \left(V_{\ydiagram{2,2}}(3) \otimes \det\ \!\!\!^{-2}\right)\boxtimes V_{\ydiagram{1}}(2).\]
No other highest-weight matrix in $\Std(\Omega)$ has this profile, and hence $m_{\Omega,{\bf \Lambda}}=1$. 
More generally, the existence of \emph{any} ${\bf M}\in \Std(\Omega)$ with $\operatorname{profile}({\bf M})={\bf \Lambda}$
establishes positive multiplicity (Corollary~\ref{cor:nonzeroness}).

A classical family is given by the
\emph{Buchsbaum--Eisenbud varieties of complexes}  
introduced in \cite{Buchsbaum.Eisenbud}  and studied by Kempf \cite{Kempf}
in the 1970s and later in the 1980s by De Concini–Strickland \cite{Strickland}. Consider a sequence of linear maps such that the composition of successive maps is zero:
\[V_1 \xrightarrow{L_1} V_2 \xrightarrow{L_2} \cdots \xrightarrow{L_{n-1}} V_n \text{\ such that } L_{i+1}L_i={\bf 0}, 1\leq i\leq n-2.\]  
We use Theorem~\ref{thm:newrule} to give a simple formula for this multiplicity-free decomposition (Corollary~\ref{thm:complexesapp}). 
Earlier, the decomposition problem and multiplicity-freeness were treated implicitly in a remark in~\cite[p.~70]{Strickland}.

Another example is the $A_3$ quiver representation space $\rep_{Q}({\bf d})$ of
${\mathbb C}^p\rightarrow {\mathbb C}^q \leftarrow
{\mathbb C}^r$. Here, Theorem~\ref{thm:newrule} computes $(GL_{r+p},GL_r\times GL_p)$-branching 
multiplicities. These multiplicities are well-known
to be \emph{Littlewood-Richardson coefficients}. In this sense, Theorem~\ref{thm:newrule} is a generalized Littlewood-Richardson
rule. The connection to Littlewood-Richardson combinatorics goes deeper, and we plan to pursue it in a sequel.

\subsection{Proof methods and organization of the paper}
The first ingredient is the set of $\Omega$-admissible matrices introduced above.
The second, developed in Section~\ref{sec:quivercrystalgraph} and forming the
technical core of this paper, is a crystal structure on this set.
This yields a crystal-theoretic reformulation of Theorem~\ref{thm:newrule} (Theorem~\ref{thm:mainprecise}), in which the multiplicities are counted by highest-weight $\Omega$-admissible matrices.
This requires developing, in Section 3, a separated-word crystal theory compatible with Stembridge--Stroomer \emph{rational insertion} 
\cite{Stembridge, Stroomer}, building on Kwon's crystal model for \emph{rational semistandard tableaux} \cite[\S~5.3, Lemma~5.8]{Kwon}, \cite[\S~3.4]{Kwon2}. 

Earlier work of Kinser--Rajchgot \cite{KinserRajchgot} reduced many
problems about quiver loci in type $A_n$ to the bipartite case; this
geometric input is reflected in the $Q$-shape and is used throughout
our proofs. In the bipartite case, quiver loci are realized as
Kazhdan--Lusztig varieties, and the Gr\"obner basis theorem of Woo and
the third author \cite{WooYong}  degenerates their
determinantal rank conditions to a monomial ideal generated by
antidiagonal products.
Kinser--Knutson--Rajchgot
\cite{KKR} combine these results to obtain a quotient formula for the
multigraded Hilbert series, i.e., the character of $\mathbb C[\Omega]$
for the maximal torus of $GL$.

Our third ingredient is to retain the fine multigrading of this
Gr\"obner degeneration, remembering the individual positions of
the $Q$-shape. Extracting the contribution in which every phantom
position occurs exactly once, and \emph{only then} specializing to torus
weights, yields the first monomial-positive formula for this character
in terms of admissible matrices (Theorem~\ref{thm:poschar}).
Consequently, as explained in Section~\ref{sec:proof-cont}, the
admissible matrices indexing this character decompose into connected
components of the crystal, from which Theorem~\ref{thm:newrule}
follows.

Thus the same set of $\Omega$-admissible matrices both indexes the positive character formula and carries
the crystal structure that converts it into the irreducible decomposition.

Section~\ref{sec:S1} contains preliminaries about quiver representations, tableaux, and crystal graphs.
Section~\ref{sec:apptocomplex} applies our results to give a simple formula for varieties of complexes.

\section{Preliminaries}\label{sec:S1}

\subsection{Quivers and representations}
We need basic notions about quiver representations found
in textbooks such as \cite{Derksenbook}. 

A \emph{quiver} $Q$ is a directed graph with vertex set $Q_0$ and arrow set $Q_1$, together with two maps
$t,h : Q_1 \to Q_0$ 
sending an arrow to its tail and head, respectively.
A \emph{representation} ${\sf V}$ of $Q$ assigns a finite-dimensional vector space $V_x$ to each $x\in Q_0$
and a linear map 
$f_e:V_{t(e)}\to V_{h(e)}, e\in Q_1$. 
The \emph{dimension vector} of ${\sf V}$ is 
$\mathbf d:Q_0\to {\mathbb Z}_{\ge 0},  \, \mathbf d(x)=\dim V_x$.
Let $\operatorname{Mat}_{m,n}$ be the space of $m\times n$ complex matrices $M$.
The space of representations with dimension vector $\mathbf d$ is the affine space
\[\rep_Q(\mathbf d) \;=\; \prod_{e\in Q_1} \mathrm{Mat}_{\mathbf d(h(e)),\,\mathbf d(t(e))}.\]
(When forming the $Q$-shape in the introduction we transpose these matrices.)

The \emph{base change group} is
$GL(\mathbf d) \;=\; \prod_{x\in Q_0} GL_{\mathbf d(x)}({\mathbb C})$,
acting on $\rep_Q(\mathbf d)$ by
\begin{equation}\label{eqn:bcgaction}
(g_x)_{x\in Q_0}\cdot (A_e)_{e\in Q_1} \;=\; \bigl(g_{h(e)}A_e g_{t(e)}^{-1}\bigr)_{e\in Q_1}.
\end{equation}

The \emph{coordinate ring} of the affine space $\rep_Q(\d)$ is the polynomial ring $\C[\rep_Q(\d)]$ generated by the coordinate functions of the matrices, which we regard as the ring of polynomial functions $\rep_Q(\d)\to \C$. The base change group acts on $\C[\rep_Q(d)]$ by pullback of functions.

\subsection{Orbits and loci; basic examples}

A \emph{quiver orbit} $\Omega^\circ$ is a $GL(\mathbf d)$-orbit in $\rep_Q(\mathbf d)$.
Its Zariski closure 
$\Omega=\overline{\Omega^\circ}\subset \rep_Q(\mathbf d)$
is a \emph{quiver locus}.
For $A_n$ quivers, $\rep_Q(\mathbf d)$ has finitely many $GL(\mathbf d)$-orbits. 
The action of the base change group $GL(\d)$ on $\C[\rep_Q(\d)]$ induces the contragredient action on $\C[\Omega]$.
By Gabriel's Theorem, $\Omega^{\circ}$ has a representation whose maps are partial permutation matrices \cite{KKR}. 

\begin{example}[Determinantal varieties]
The space $\operatorname {Mat}_{m,n}=\rep_Q(\d)$ 
where $Q$ is the $A_2$ quiver ${\mathbb C}^m{\longrightarrow}
{\mathbb C}^n$ and ${\bf d}=(m,n)$. 
It has an action of 
${GL}(\d)=GL_m\times GL_n$ by row and column operations. 
The orbits $\Omega^{\circ}_k$ 
consist of matrices of a fixed rank $k$ for $0\leq k\leq \min(m,n)$; they are the 
$GL(\d)$-orbit of the partial identity matrix with $k$ many $1$'s.
Their
closures, the \emph{classical determinantal varieties}, 
${\Omega}_k$ consist of matrices of rank $\leq k$. 
\end{example}

\begin{example}[Varieties of complexes]\label{exa:BE-details}
As a basic example of the varieties of complexes studied in \cite{Strickland}, consider
\[X=\{(A,B)\in {\sf Mat}_{2,2}^2\mid BA=0\}.\]
$X\subset \rep_Q(\d)$ where $Q$ is the $A_3$ quiver 
${\mathbb C}^2{\longrightarrow} {\mathbb C}^2 {\longrightarrow} {\mathbb C}^2$,
and ${\bf d}=(2,2,2)$. 
Now, ${GL}(\d)=GL_2\times GL_2\times
GL_2$ acts on $X$ since 
\[(g_1,g_2,g_3)\cdot(A,B)=(g_2 A g_1^{-1},
g_3 B g_2^{-1}) \implies (g_3 B g_2^{-1})(g_2A g_1^{-1})=0.\] 
This action preserves
the ranks of $A$ and $B$, and $X$ has three irreducible components. The component $\Omega$ where
we impose $\rank(A),\rank(B)\leq 1$, is the closure of the $GL$-orbit of 
$A=\left[\begin{smallmatrix}
0 & 0 \\
0 & 1
\end{smallmatrix}\right], B=\left[\begin{smallmatrix} 1 & 0 \\ 0 & 0 \end{smallmatrix}\right]$. 
\end{example}

\begin{example}[Bipartite quiver]
Let 
\[\Omega^{\circ}=\{(A,B)\in {\sf Mat}_{2,2}^2 \mid \operatorname{rank}(A)={\operatorname{rank}}(B)=\dim(\operatorname{im}(A)+\operatorname{im}(B))=1\}.\]
$\Omega^{\circ}$ is a ${GL}(\d)=GL_2\times GL_2\times GL_2$ orbit in the $A_3$ quiver
${\mathbb C}^2\rightarrow {\mathbb C}^2 \leftarrow
{\mathbb C}^2$. Here the ${GL}(\d)$-action is by
\[(g_1,g_2,g_3)\cdot(A,B)=(g_2 A g_1^{-1}, g_2 B g_{3}^{-1}).\]
Now, $\Omega^{\circ}$ is the $GL(\d)$-orbit of the pair $\left(\left[\begin{smallmatrix} 1 & 0 \\ 0 & 0 \end{smallmatrix}\right],
\left[\begin{smallmatrix} 1 & 0 \\ 0 & 0 \end{smallmatrix}\right]\right)\in \Omega^{\circ}$. Taking the closure, 
\[\Omega=\{(A,B)\in {\sf Mat}_{2,2}^2 \mid \operatorname{rank}(A), {\operatorname{rank}}(B), \dim(\operatorname{im}(A)+\operatorname{im}(B))\leq 1\}.\]
The work of \cite{KinserRajchgot,KKR}, as explained in Section~\ref{subsec:theredux}, reduces the study of the 
character of loci for arbitrary $A_n$ quivers
to the bipartite case. 
\end{example}

\subsection{RSK insertion} \label{sec:RSKinsertion}
Ordinary words are inserted by RSK row or column insertion producing semistandard Young tableaux; a reference is \cite{Fulton-book}. 

\begin{definition}[Row insertion] Let 
$\bs{a} = a_1\le a_2\le\cdots\le a_k$ 
be a sequence of integers, and let $b\in {\mathbb N}$. The \emph{row insertion} 
$\bs{a}\leftarrow b$
of $b$ into $\bs{a}$ is the pair $(\bs{a}',b')$ where $\bs{a'}$ is $\bs{a}$ with $a_j$ replaced with $b$ for the smallest value of $j$ with $b<a_j$, and $b'=a_j$. If $a_k\le b$, $\bs{a}'$ is the sequence $a_1\le \cdots \le a_k\le b$, and $b' = \varnothing$. Also, define $\bs{a}\leftarrow\varnothing = (\bs{a},\varnothing)$ and $\varnothing\leftarrow b = (b,\varnothing)$.
\end{definition}

\begin{definition}[Column insertion] Let 
$\bs{a} = a_1< a_2<\cdots< a_k$ 
be a sequence of integers, and let $b\in {\mathbb N}$. The \emph{column insertion} 
$b\rightarrow \bs{a}$
of $b$ into $\bs{a}$ is the pair $(\bs{a}',b')$ where $\bs{a'}$ is $\bs{a}$ with $a_j$ replaced with $b$ for the smallest value of $j$ with $b\leq a_j$, and $b'=a_j$. If $a_k< b$, $\bs{a}'$ is the sequence $a_1< \cdots < a_k< b$, and $b' = \varnothing$. Also, define $\varnothing \rightarrow \bs{a}= (\bs{a},\varnothing)$ and $b\rightarrow \varnothing = (b,\varnothing)$.
\end{definition}

\begin{definition} \label{def:rowiterate}
If $T$ is semistandard, the $i$th row of $T$ is an increasing sequence $\bs{a}_{(i)}$. The \emph{row insertion} $T \leftarrow b$ is the tableau with rows $\bs{a}_{(i)}'$, where $b_0 = b$ and $\bs{a}_{(i)}'$ is recursively defined as $(\bs{a}_{(i)}',b_i) = \bs{a}_{(i)}\leftarrow b_{i-1}$. Similarly, 
one defines the \emph{column insertion} $b\rightarrow T$.
\end{definition}

\begin{example}
\ytableausetup{boxsize=.8em}
 Let $T=\vcenter{\hbox{\begin{ytableau}
 1 & 1 & 2 & 2 & 3\\
 2& 3 & 3 \\
 4 & 4 & 5
 \end{ytableau}}}$ and $b=1$. Then $T\leftarrow b$ is computed by the following sequence where the entry to be bumped is marked:
\[T=\vcenter{\hbox{\begin{ytableau}
 1 & 1 & {\bf 2} & 2 & 3\\
 2& 3 & 3 \\
 4 & 4 & 5
 \end{ytableau}}} \mapsto \vcenter{\hbox{\begin{ytableau}
 1 & 1 & 1 & 2 & 3\\
 2& {\bf 3} & 3 \\
 4 & 4 & 5
 \end{ytableau}}}\mapsto
 \vcenter{\hbox{\begin{ytableau}
 1 & 1 & 1 & 2 & 3\\
 2& 2 & 3 \\
 {\bf 4} & 4 & 5
 \end{ytableau}}}\mapsto
 \vcenter{\hbox{\begin{ytableau}
 1 & 1 & 1 & 2 & 3\\
 2& 2 & 3 \\
 3 & 4 & 5 \\ 
 4
 \end{ytableau}}} = T\leftarrow 1.
 \]
If instead we compute $b\rightarrow T$, the result is: 
\[T=\vcenter{\hbox{\begin{ytableau}
 {\bf 1} & 1 & 2 & 2 & 3\\
 2& 3 & 3 \\
 4 & 4 & 5
 \end{ytableau}}} \mapsto \vcenter{\hbox{\begin{ytableau}
 1 & {\bf 1} & 2 & 2 & 3\\
 2& 3 & 3 \\
 4 & 4 & 5
 \end{ytableau}}}\mapsto
 \vcenter{\hbox{\begin{ytableau}
 1 & 1 & {\bf 2} & 2 & 3\\
 2& 3 & 3 \\
  4 & 4 & 5
 \end{ytableau}}}\mapsto
 \vcenter{\hbox{\begin{ytableau}
 1 & 1 & 1 & {\bf 2} & 3\\
 2& 3 & 3 \\
 4 & 4 & 5 
 \end{ytableau}}}
 \mapsto
 \vcenter{\hbox{\begin{ytableau}
 1 & 1 & 1 & 2 & {\bf 3}\\
 2& 3 & 3 \\
 4 & 4 & 5 
 \end{ytableau}}}
 \mapsto
 \vcenter{\hbox{\begin{ytableau}
 1 & 1 & 1 & 2 & 2 & 3\\
 2& 3 & 3 \\
 4 & 4 & 5 
 \end{ytableau}}}
 = 1\!\rightarrow\!T.
 \]

\end{example}

\begin{definition}\label{def:insis}
Let $\bs{w} = (w_1,\ldots,w_k)\in {\mathbb N}^k$. The \emph{row insertion tableau} of $\bs w$ is
\[\ins_r(\bs w)=
((\cdots(\varnothing \leftarrow w_1) \leftarrow w_2) \cdots )\leftarrow w_k\]
and the \emph{column insertion tableau} of $\bs w$ is
\[\ins_c(\bs w)=
w_1 \rightarrow (\cdots (w_{k-1} \rightarrow (w_k \rightarrow \varnothing))).\]
\end{definition}

In fact, $\ins_r(\bs{w}) = \ins_c(\bs{w})$ for any $\bs{w}\in {\mathbb N}^k$~\cite[\S A.2]{Fulton-book}. However, this will no longer be true in the generalization to rational tableaux, which we come to in the next subsection.

\subsection{Stembridge's rational tableaux} \label{sec:stembridge-insertion}
 
A \emph{rational tableau}  \cite{Stembridge}
is a pair of semistandard tableaux $(U,V)$, satisfying the \emph{Stembridge constraint}, i.e., 
for each $i$, the total number of entries in $[1,i]$
appearing in the first columns of $U$ and $V$ is at most $i$. Denote this rational tableau by $U\div V$. 
Let $\mu,\nu$ be the respective partition shapes of $U,V$; let $\mu\div\nu$ be the \emph{rational shape} of $U\div V$.

Suppose $U$ and $V$ have entries from $[d]=\{1,2,\ldots,d\}$.
The \emph{shift operator} is
\[\operatorname{shift}(U \div V) = U' \div V',\] 
where $V'$ is $V$ with the leftmost column $c$ removed and $U'$ is the 
tableau obtained from $U$ by adding a column to the left consisting of the complement in $[d]$ of the entries of $c$.
The Stembridge constraint on $U\div V$ guarantees that the new leftmost column of $U'$ is compatible with semistandardness. 
For $k\gg 0$, $\operatorname{shift}^k(U \div V)$ is a semistandard tableau of a straight shape.\footnote{See \cite[Proof of Proposition~2.4]{Stembridge}. Combinatorially, the shift operation corresponds to 
tensoring with the determinant representation.}

For rational tableaux, Stembridge \cite{Stembridge} defines row insertion by
\[
(U\div \varnothing)\leftarrow i := (U\leftarrow i)\div \varnothing
\qquad\text{ and }\qquad
(\varnothing\div V)\leftarrow\barindex{i} := \varnothing\div (V\leftarrow i).
\]
and in general by reduction to ordinary row insertion as follows:
\begin{equation}\label{eqn:May29ddd}
(U\div V) \leftarrow i:= \operatorname{shift}^{-k}(\operatorname{shift}^k(U\div V)\leftarrow i), \qquad k\gg 0,
\end{equation}
\begin{equation}\label{eqn:May29eee}
(U\div V) \leftarrow \barindex{i}:= \operatorname{shift}^k(\operatorname{shift}^{-k}(U\div V)\leftarrow \barindex{i}), \qquad k\gg 0,
\end{equation}
and these definitions are independent of $k$ for large enough $k$ (see Theorem~\ref{thm:Steminsprops} below). 

If ${\bs w}=(w_1,\ldots,w_k)$ is a word consisting of barred and unbarred letters, define
\[
\ins_r({\bs w})
=
(\cdots((\emptyset\leftarrow w_1)\leftarrow w_2)\cdots)\leftarrow w_k,
\]
 using
 \eqref{eqn:May29ddd} for unbarred letters and \eqref{eqn:May29eee} for barred letters. 
 
Similarly, for rational tableaux, Stroomer \cite{Stroomer} defines column insertion. Let
\[
i\rightarrow(U\div \varnothing) := (i\rightarrow U)\div \varnothing
\qquad\text{ and }\qquad
\barindex{i}\rightarrow(\varnothing\div V) := \varnothing\div (i\rightarrow V).
\]
In general, the definition is by reduction to ordinary column insertion:
\begin{equation}\label{eqn:May26aaa}
i\rightarrow(U\div V) := \operatorname{shift}^{-k}(i\rightarrow \operatorname{shift}^k(U\div V)), \qquad k\gg 0,
\end{equation}
\begin{equation}\label{eqn:May26bbb}
\barindex{i}\rightarrow(U\div V) := \operatorname{shift}^k(\barindex{i}\rightarrow \operatorname{shift}^{-k}(U\div V)), \qquad k\gg 0.
\end{equation}

\begin{example}
Suppose $d=4$ and $U\div V=\begin{ytableau} 1 \end{ytableau}
\div
\begin{ytableau} 2 &3 \end{ytableau}$. To compute $2\rightarrow U\div V$ we first shift twice to obtain a 
straight-shape semistandard tableau: 
\[\operatorname{shift}^2(U\div V)=\vcenter{\hbox{\begin{ytableau} 1 & 1 & 1 \\ 2 & 3 \\ 4 & 4\end{ytableau}}} \implies
2\rightarrow  \operatorname{shift}^2(U\div V) =\vcenter{\hbox{\begin{ytableau} 1 & 1 & 1 \\ 2 & 2 & 3 \\ 4 & 4\end{ytableau}}}.\]
Shifting back we conclude
$2\rightarrow U\div V = \vcenter{\hbox{\begin{ytableau} 1 \\ 3 \end{ytableau}}} \div \vcenter{\hbox{\begin{ytableau} 3 &3 \end{ytableau}}}$.
\end{example}
If $\bs{w} = (w_1,\ldots,w_k)$ is a word consisting of unbarred and barred entries, define 
\[\ins_c(\bs w)=
w_1 \rightarrow (\cdots (w_{k-1} \rightarrow (w_k \rightarrow \varnothing)))\]
 using
 \eqref{eqn:May26aaa} for unbarred letters and \eqref{eqn:May26bbb} for barred letters. 

Both row and column rational insertion have descriptions not invoking shifting. As we will not make use of these
other descriptions here, we refer the reader to  \cite[\S~3]{Stembridge} and \cite[\S~4]{Stroomer}.

\subsection{Crystals}
Our reference for seminormal crystals is \cite[Chapter~2]{Bump.Schilling}.

\begin{definition} \label{def:crystal}
Fix a weight lattice $\Lambda$ and let $\Phi$ be the corresponding root system with index set $I$ and simple roots $\Delta = \{\alpha_i | i\in I\}$. A \emph{(seminormal) crystal} of type $\Lambda$ is a nonempty set $\mathcal{B}$ along with maps
\[
e_i,f_i: \mathcal{B}\to \mathcal{B}\sqcup \{\o\},\;\; (i\in I) \qquad \text{and}\qquad 
\wt: \mathcal{B}\to\Lambda \qquad \text{such that:}
\]
\begin{enumerate}
\item[(a)] If $x,y\in\mathcal{B}$, then $e_i(x)=y$ if and only if $f_i(y) = x$, and in this case $\wt(y) = \wt(x) + \alpha_i$.
\item[(b)] For all $x\in\mathcal{B}$, the quantities
\[
\phi_i(x) := \max\{ k\in\mathbb{Z}_{\ge 0} \;|\; f_i^k(x)\ne \o\} \quad \epsilon_i(x) := \max\{ k\in\mathbb{Z}_{\ge 0} \;|\; e_i^k(x)\ne \o\}
\]
are finite and satisfy $\phi_i(x) - \epsilon_i(x) = \langle \wt(x), \alpha_i^\vee\rangle$.
\end{enumerate}
\end{definition}

In Definition~\ref{def:crystal},
$\wt$ is the \emph{weight map} for $\mathcal{B}$, and $\phi_i(x)$ and $\epsilon_i(x)$ are the \emph{root string lengths}. The $e_i$ are \emph{raising operators} and the $f_i$ are \emph{lowering operators}; together, they are the \emph{crystal operators} for $\mathcal{B}$.
In this paper, we work with crystals for products of general linear groups. In particular if $G=GL_d$ the lattice is ${\mathbb Z}^d$ with $A_{d-1}$ root system.

The \emph{crystal graph} is the labelled directed graph consisting of a vertex for each element of $\mathcal{B}$ and an edge
\tikz[baseline=(A.base)] {
    \node (A) {$x$};
    \node (B) at (1.2,0) {$y$};
    \draw[-{Stealth[scale=1.2]}] (A) -- node[above] {\scriptsize $t_i$} (B);
}
whenever $t_i(x) = y$ for $t_i\in\{e_i,f_i\}$. $\mathcal{B}$ is \emph{connected} if its graph is connected.

Let $\mathcal{B}'$ be another crystal associated to $\Lambda$ with crystal operators $\{e_i',f_i' \;|\; i\in I\}$ and weight map $\wt':\mathcal{B}'\to\Lambda$. 
\begin{definition}
A \emph{crystal isomorphism} is a bijection $\psi: \mathcal{B}'\to\mathcal{B}$ such that
\begin{equation} \label{eq:crystal-morphism}
e_i\circ\psi = \psi \circ e_i', \qquad\qquad f_i\circ\psi = \psi \circ f_i', \qquad\qquad \wt' = \wt\circ\psi.
\end{equation}
\end{definition}
As usual, for any map $\psi:\mathcal B'\to\mathcal B$, we extend $\psi$ to a map
$\mathcal B'\sqcup\{\o\}\to\mathcal B\sqcup\{\o\}$ by setting $\psi(\o)=\o$.
Automatic from \eqref{eq:crystal-morphism} is that crystal isomorphisms preserve root string lengths: 
\begin{equation}\label{eqn:July29aaa}
\phi_i(\psi(x))=\phi_i'(x), \qquad \epsilon_i(\psi(x))=\epsilon_i'(x).
\end{equation}

\begin{definition}
A bijection $\psi:\mathcal{B}'\to\mathcal{B}$ is a \emph{shifted crystal isomorphism} if 
\begin{equation} \label{eq:shifted-crystal-morphism}
e_i\circ\psi = \psi \circ e_i', \qquad\qquad f_i\circ\psi = \psi \circ f_i', \qquad\qquad \wt' = \wt\circ\;\psi - H.
\end{equation}
for some fixed $H\in\span(\Lambda)$ such that $\langle H, \alpha_i^\vee\rangle = 0$ for all $i\in I$. $H$ is the \emph{shift} of $\psi$. 
\end{definition}
For $H\neq \vec 0$, a shifted crystal
isomorphism only exists if the rank of $\Lambda$ exceeds $\Phi$'s rank.

Now suppose ${\mathcal B}$ is a crystal but
${\mathcal B}'$ is not known to satisfy (a) nor (b) of Definition~\ref{def:crystal}. The following proposition allows us to simultaneously prove ${\mathcal{B}}'$ is a crystal and
that $\mathcal{B}'\cong\mathcal{B}$; we use it repeatedly in this paper. 

\begin{proposition} \label{prop:inj-implies-bij}
Let $\mathcal{B}$ be a connected crystal with root system $\Phi$, index set $I$, weight lattice $\Lambda$, crystal operators $\{e_i,f_i \;|\; i\in I\}$, and weight map $\wt:\mathcal{B}\to \Lambda$.
Let $\mathcal{B}'$ be a nonempty set along with maps 
\[
e_i',f_i': \mathcal{B}'\to \mathcal{B}'\sqcup \{\o\},\;\; (i\in I) \qquad \text{and}\qquad 
\wt': \mathcal{B}'\to\Lambda.
\]
If $\psi:\mathcal{B}'\to\mathcal{B}$ is an injective map satisfying~\eqref{eq:shifted-crystal-morphism},
then ${\mathcal B}'$ is a crystal and $\psi:{\mathcal B}'\to {\mathcal B}$ is a shifted crystal 
isomorphism with shift $H$.
\end{proposition}

\begin{proof}
First, we claim that $\psi$ is surjective (and therefore bijective). 
If $\psi$ is not surjective, then $\text{im}(\psi)$ is a nonempty proper subset of $\mathcal B$. 
Since $\mathcal B$ is connected, there is an edge between a vertex $b=\psi(c)\in\operatorname{im}(\psi)$ and a vertex not in $\text{im}(\psi)$. 
Thus, for some $i\in I$ and $u_i\in\{e_i,f_i\}$, we have
$u_i(b)\notin \text{im}(\psi)$ and $u_i(b)\neq \o$. 
Let $u_i'$ denote the corresponding operator on ${\mathcal B}'$. 
If $u_i'(c)\neq \o$, then the intertwining relation in \eqref{eq:shifted-crystal-morphism} gives
\[
u_i(b)=u_i(\psi(c))=\psi(u_i'(c))\in \operatorname{im}(\psi),
\]
a contradiction. If $u_i'(c)=\o$, the same computation shows $u_i(b)=\o$, again a contradiction.

Next we show that $\mathcal{B'}$ satisfies (a) and (b) from Definition~\ref{def:crystal}
and is therefore a crystal. For (a), suppose that $e_i'(x)=y\in\mathcal{B}'$. Then,
\[
\psi(y)=\psi(e_i'(x))=e_i(\psi(x)).
\]
Applying $f_i$ and using that $\mathcal B$ is a crystal,
$f_i(\psi(y))=\psi(x)$.
By \eqref{eq:shifted-crystal-morphism},
\[
\psi(f_i'(y))=f_i(\psi(y))=\psi(x).
\]
Since $\psi$ is injective, $f_i'(y)=x$.
Similarly, if $f_i'(y)=x$, then $e_i'(x)=y$.
By~\eqref{eq:shifted-crystal-morphism}, if $e_i'(x) = y$, 
\[
\wt'(y)+H = \wt(\psi(y)) = \wt(\psi(e_i'(x))) = \wt(e_i(\psi(x))) = \wt(\psi(x)) + \alpha_i = \wt'(x) + \alpha_i+H.
\]
Hence (a) holds. 

For (b), since $\psi$ is bijective and we assume \eqref{eq:shifted-crystal-morphism}, the same
rationale for \eqref{eqn:July29aaa} holds here. In particular, 
$\phi_i'(x), \epsilon_i'(x)<\infty$. Hence,
\[
\phi_i'(x)-\epsilon_i'(x)
=\phi_i(\psi(x))-\epsilon_i(\psi(x))  
=\langle \wt(\psi(x)),\alpha_i^\vee\rangle 
=\langle \wt(\psi(x))-H,\alpha_i^\vee\rangle 
=\langle \wt'(x),\alpha_i^\vee\rangle,
\]
where the penultimate equality is because $\langle H, \alpha_i^\vee\rangle = 0$. Thus (b) is satisfied.

This shows that $\mathcal{B}'$ is a crystal, and since $\psi$ is a bijective map between crystals satisfying~\eqref{eq:shifted-crystal-morphism}, it is a shifted crystal isomorphism.
\end{proof}

\section{Separated-word crystals and rational insertion}\label{sec:thetechproofs}

\subsection{Crystal operators on rational tableaux} \label{sec:crystal-ops-rational-tableaux}

For any rational shape $\gamma$ and positive integer $d$, let $\mathcal{B}_\gamma := \mathcal{B}^{(d)}_\gamma$ be the set of rational tableaux of shape $\gamma = \lambda\div \mu$ with entries in $[d]=\{1,2,\ldots,d\}$. Given 
a rational tableau $T=U\div V$, the \emph{weight} of $T$,
denoted $\operatorname{wt}(T)$, is the vector in ${\mathbb Z}^d$ with $\operatorname{wt}(T)_i$ equal to
the number of $i$'s in $U$ minus the number of $i$'s in $V$.

If $\mu=\varnothing$, then $\gamma = \lambda$ is a straight-shape and ${\mathcal B}_{\lambda}$
is the crystal of the polynomial  $\operatorname{GL}_d$-representation with highest weight~$\lambda$; see, e.g., \cite[Chapters~2,3]{Bump.Schilling}. Here $I=[d-1]$, $\Phi$ is the $A_{d-1}$-root system with $\Delta=\{\varepsilon_i-\varepsilon_{i+1} | i\in I\}$, and $\Lambda$ is the $GL_d$ weight lattice ${\mathbb Z}^d$. 
The crystal operators $\tabe_i,\tabf_i$ on tableaux are defined as follows.
Construct the \emph{row word} of $U$ by reading $U$ along rows, from left to right, and bottom to top. Then assign
\[i \mapsto \texttt{)}, \qquad
i+1 \mapsto \texttt{(}.\]
An $i+1$ is replaced
by $i$ if it corresponds to the leftmost unpaired {\tt (} in the row word. If no such unpaired {\tt (} exists, $\o$ is returned. Similarly, the lowering operator $\tabf_i$ ($1\leq i < d$) replaces the $i$ associated to the rightmost unpaired {\tt )} (if it exists) by $i+1$
 or else returns $\o$. 

More generally, for any rational shape $\gamma$, $\mathcal{B}_\gamma$ admits a crystal structure 
associated to the rational $\operatorname{GL}_d$-representation with highest weight $\gamma$; see Kwon~\cite[\S~5.3, Lemma~5.8]{Kwon} and
\cite[\S~3.4]{Kwon2}. The shift map (\S~\ref{sec:stembridge-insertion}) is a bijection between $\mathcal{B}_\gamma$ and $\mathcal{B}_{\operatorname{shift}(\gamma)}$, where
\[\operatorname{shift}(\gamma) = (\lambda + 1^{d-\ell(\mu)})\div (\mu - 1^{\ell(\mu)}).\]
One extends the crystal operators to rational tableaux by
\begin{equation}
\label{eqn:Apr30ccc}
\tabe_i(U\div V) := \operatorname{shift}^{-k}(\tabe_i(\operatorname{shift}^k(U \div V))), \qquad k\gg 0,
\end{equation}
and similarly for $\tabf_i$. This definition is independent of the choice of
sufficiently large $k$ since 
for straight shape tableaux the crystal operators commute with the shift map:
\begin{lemma}\label{lemma:June18zzz}
For straight-shape tableaux,
\[\tabe_i \circ \operatorname{shift}=\operatorname{shift}\circ\tabe_i,
\qquad \tabf_i \circ \operatorname{shift} = \operatorname{shift}\circ\tabf_i.\]
\end{lemma}
\begin{proof}
If $T$ is a straight-shape semistandard tableau then $\operatorname{shift}(T)$
adds a full column of $1,2,\ldots,d$.
In the $i,i+1$ bracket assignment, this adds a {\tt (} at the start of row $i+1$ and a {\tt )} at the start of row
$i$. Column-strictness forces every $i+1$ in row $i+1$ to sit directly below an $i$ in row $i$, so the {\tt (}'s
from row $i+1$ and the {\tt )}'s beginning row $i$ cancel completely, both before and after the shift.
Hence the set of unpaired brackets determining the crystal
operators is unaffected.
\end{proof}

\begin{proposition} \label{prop:tableau-shift-isomorphism}
For any $k\in\mathbb Z$, the map $\operatorname{shift}^k$ is a shifted
crystal isomorphism from $\mathcal{B}_\gamma$ to
$\mathcal{B}_{\operatorname{shift}^k(\gamma)}$ with shift
$H=(k,\ldots,k)$.
\end{proposition}
\begin{proof}
The case $k=1$ follows from Lemma~\ref{lemma:June18zzz} and
Proposition~\ref{prop:inj-implies-bij}. Since the inverse of a shifted
crystal isomorphism with shift $H$ is a shifted crystal isomorphism with
shift $-H$, the case $k=-1$ holds. The general case follows by iteration.
\end{proof}

The following fact is then immediate:
\begin{corollary}[Highest-weight rational tableaux]\label{cor:superlowertab}
Let $T=U\div V\in\mathcal B_{\lambda\div\mu}^{(d)}$.
Then $T$ is highest-weight if and only if $U$ is
supersemistandard and $V$ is the lowest-weight tableau of shape
$\mu$. Equivalently, every column of $V$ of height $h$ is
$d-h+1,d-h+2,\ldots,d$ from top to bottom.
\end{corollary}

\subsection{Crystal operators on words}\label{subsec:CryStembridge}
Fix $d\in {\mathbb Z}_{>0}$. Let ${\mathcal W}_d^{\pm}$ be the set of words in the alphabet
$[d]\cup [\barindex d]$ where $[\barindex d]=\{\barindex{1},\ldots,\barindex{d}\}$. 
For ${\bs w}=w_1\cdots w_k\in {\mathcal W}_d^{\pm}$, let $\ell({\bs w}) := k$ be its \emph{length}. The \emph{weight} of ${\bs w}$ is the vector $\operatorname{wt}({\bs w})\in{\mathbb Z}^d$ whose $i$-th component equals the number of $i$'s minus the number of $\barindex{i}$'s in ${\bs w}$.

Let ${\mathcal W}_{d}^{-+}$ be words in ${\mathcal W}_d^{\pm}$ that are a barred word followed by an
unbarred word. Let ${\mathcal W}_d^{+-}$ denote the set of words that are an unbarred word followed
by a barred word. If $\bs{w}\in {\mathcal W}_{d}^{\operatorname{sep}}:= {\mathcal W}_d^{-+}\cup {\mathcal W}_d^{+-}$,
we call it \emph{separated}. Furthermore, let ${\mathcal W}_{d}$ be the set of unbarred words and ${\mathcal W}_{d}^-$ be the set of barred words.
Each of these sets inherits the above weight function.

The \emph{modified word} of a separated word is obtained by reversing the barred part. Thus, if ${\bs n}= \barindex{n_1},\ldots, \barindex{n_\ell}$ and ${\bs p}=p_1,\ldots, p_k$, then 
\[\operatorname{rev}_{-}({\bs n}{\bs p}) = \barindex{n_\ell},\ldots, \barindex{n_1},p_1,\ldots, p_k, \qquad 
\operatorname{rev}_{-}({\bs p}{\bs n}) = p_1,\ldots, p_k,\barindex{n_\ell},\ldots, \barindex{n_1}.\] 

We define \emph{rational crystal operators} on ${\mathcal W}_d^{-+}$ and 
${\mathcal W}_d^{+-}$ extending the  
Kashiwara crystal operators on ${\mathcal W}_d$. 
Take the modified word $\operatorname{rev}_{-}({\bs w})$, and associate the brackets:
\[
i\mapsto {\texttt{)}},\qquad i{+}1\mapsto {\texttt (},\qquad \barindex{i}\mapsto {\texttt (},\qquad \barindex{i{+}1}\mapsto 
{\texttt )}.
\]
(If ${\bs w}\in {\mathcal W}_d$, this is the usual Kashiwara parenthesization.)
To define $\worde_i$, locate the leftmost unpaired {\tt (}; if none exists, $\worde_i({\bs w})=\o$.
Change the corresponding entry in $\operatorname{rev}_{-}({\bs w})$: either $i+1$ is changed to an $i$ or $\barindex{i}$ is changed to a $\barindex{i+1}$. The result is 
$\operatorname{rev}_{-}({\bs w}')$ for some ${\bs w}'\in {\mathcal W}_d^{-+}$ (respectively,
${\mathcal W}_d^{+-}$) if
${\bs w}\in {\mathcal W}_d^{-+}$ (respectively, ${\mathcal W}_d^{+-}$). We set $\worde_i({\bs w}) := {\bs w}'$.
Now, $\wordf_i$ is similarly defined, where we locate
the rightmost unpaired {\tt )}, if it exists.
If no such bracket exists, $\wordf_i({\bs w})=\o$. In $\operatorname{rev}_{-}({\bs w})$ this either changes an $i$  to an $i+1$ or an $\barindex{i+1}$ to a $\barindex{i}$.
The resulting word is again of the form
$\operatorname{rev}_{-}(\bs w')$, and $\wordf_i(\bs w):=\bs w'$.

\begin{theorem} \label{thm:stembridge-word-crystals}
\begin{enumerate}
\item[(I)] $({\mathcal W}^{-+}_d, \operatorname{wt}, \{\worde_i, \wordf_i\})$ is a crystal. If $\mathcal{C}\subseteq {\mathcal W}^{-+}_d$  is a connected component, 
$\ins_r:{\mathcal C}\to {\mathcal B}_{\gamma}$
is a crystal isomorphism, where $\gamma$ is the rational shape of $\ins_r({\bs w})$ independent of ${\bs w}\in {\mathcal C}$.
\item[(II)] The analogous statement holds for ${\mathcal W}_d^{+-}$, replacing
$\ins_r({\bs w})$ by $\ins_c({\bs w})$.
\end{enumerate}
\end{theorem}

\begin{example}
The word $\barindex{1}1$ forms its own connected component in ${\mathcal W}_{2}^{-+}$
and it is isomorphic to ${\mathcal B}_{\varnothing}$ as $\operatorname{ins}_r(\operatorname{\barindex{1}1})=\varnothing$,
agreeing with  Theorem~\ref{thm:stembridge-word-crystals}(I).
Now $1\barindex{1}$ is in the connected component ${\mathcal C}$
$1\barindex{2} \xrightarrow{\wordf_1} 1\barindex{1} \xrightarrow{\wordf_1} 2\barindex{1}$
in ${\mathcal W}_2^{+-}$. While $\ins_r(1\barindex{1})=\varnothing$, evidently
${\mathcal C}\not\cong {\mathcal B}_{\varnothing}$. Rather ${\mathcal C}\cong 
{\mathcal B}_{\ytableausetup{boxsize=.4em} \ydiagram{1}\div \ydiagram{1}}$ and 
$\ins_c(1\barindex{1})=\ytableausetup{boxsize=.9em}\begin{ytableau}2\end{ytableau}
\div\begin{ytableau}{2}\end{ytableau}$, consistent with Theorem~\ref{thm:stembridge-word-crystals}(II).
\end{example}

In \S~\ref{sec:proof32}, we deduce Theorem~\ref{thm:stembridge-word-crystals} using 
Proposition~\ref{prop:inj-implies-bij}, the material of \S~\ref{sec:ordinification} and the well-known fact:
\begin{theorem}\label{thm:basicwordsarecrystal}
$(\mathcal W_d,\operatorname{wt},\{\worde_i,\wordf_i\})$ is a
$\operatorname{GL}_d$-crystal. Moreover, if $\mathcal C\subseteq
\mathcal W_d$ is a connected component, then
$\operatorname{ins}_r:\mathcal C\longrightarrow\mathcal B_\lambda$
is a crystal isomorphism, where $\lambda$ is the shape of
$\operatorname{ins}_r(w)$, independent of $w\in\mathcal C$.
\end{theorem}

\subsection{Ordinification}\label{sec:ordinification}
 We show that each connected component of 
${\mathcal W}^{-+}_d$ and ${\mathcal W}_d^{+-}$ is shifted crystal isomorphic to a connected component of 
${\mathcal W}_d$. Given $ \barindex{i}\in[\barindex{d}]$, the \emph{complement} of $\barindex{i}$ is the word $\barindex{i}^c := d \ d-1\ \cdots \ i+1 \ i-1 \ \cdots 1$; 
the complement of $i\in[d]$ is $i^c := \barindex{d} \ \barindex{d-1} \ \cdots\ \barindex{i+1}\ 
\barindex{i-1} \ \cdots \ \barindex{1}$. The \emph{reverse complement} of ${\bs w}$ is 
\[
\operatorname{rev^c}({\bs w}) = w_{\ell({\bs w})}^c \cdots w_1^c \in {\mathcal W}_d
\]
formed by reversing ${\bs w}$ and then replacing each entry $\barindex{i}$ by $\barindex{i}^c$ and $i$ by $i^c$. 

To ${\bs w} = {\bs n}{\bs p}\in {\mathcal W}_d^{-+}$,
the \emph{ordinification} of ${\bs w}$ is ${\bs w}^+ := \operatorname{rev^c}({\bs n}){\bs p} \in {\mathcal W}_d$,
and for ${\bs w} = {\bs p}{\bs n}\in {\mathcal W}_d^{+-}$,
it is ${\bs w}^+ := {\bs p}\operatorname{rev^c}({\bs n}) \in {\mathcal W}_d$.
Ordinification commutes with $\worde_i, \wordf_i$:

\begin{lemma} \label{lem-dual-words-crystals}
For any ${\bs w}\in {\mathcal W}_d^{-+}\cup {\mathcal W}_d^{+-}$ and any $i$,
\[\worde_i({\bs w})^+ = \worde_i({\bs w}^+) 
\quad \text{and} \quad
\wordf_i({\bs w})^+ = \wordf_i({\bs w}^+) 
\]
\end{lemma}
\begin{proof}
Let $\operatorname{parens}_i({\bs w})$ be the parenthesis sequence for ${\bs w}$ and $i$.
We claim that $\operatorname{parens}_i({\bs w}^+)$ is $\operatorname{parens}_i({\bs w})$ with some inserted contiguous {\tt ()}. Let $\barindex{j}\in{\bs n}$. We have three cases:
\begin{itemize}
\item $j = i$:
$
\operatorname{parens}_i(\barindex{i}^c) = \operatorname{parens}_i(d,\ldots,i+1,i-1,\ldots,1) = ``\text{{\tt (}}" = \operatorname{parens}_i(\barindex{i});
$
\item $j = i+1$:
$
\operatorname{parens}_i(\barindex{i+1}^c) = \operatorname{parens}_i(d,\ldots,i+2,i,\ldots,1) = ``\text{{\tt )}}" = \operatorname{parens}_i(\barindex{i+1});
$
\item $j \ne i, i+1$:
$
\operatorname{parens}_i(\barindex{j}^c) = \operatorname{parens}_i(d,\ldots,j+1,j-1,\ldots,1) = ``\text{\tt ()}"
$ but $\operatorname{parens}_i(\barindex{j}) = \emptyset$.
\end{itemize}
Thus, either $\operatorname{parens}_i(\barindex{j}^c) = \operatorname{parens}_i(\barindex{j})$ or the former is 
$ ``\text{\tt ()}"$ and the latter is $\emptyset$. Therefore, $\operatorname{parens}_i(\operatorname{rev^c}({\bs n}))$ is $\operatorname{parens}_i(\operatorname{rev}_{-}({\bs n}))$ 
with added ``\text{\tt ()}", as desired. The lemma follows.
%
\end{proof}

Let 
${\mathcal W}_{\ell,d}^{-+}=\{{\bs w}={\bs n}{\bs p}\in {\mathcal W}_d^{-+}: \ell({\bs n})=\ell\}$, 
and similarly for ${\mathcal W}_{\ell,d}^{+-}$. 
\begin{lemma}\label{lemma:May13aaa}
\begin{itemize}
\item[(I)]
For fixed $\ell\in {\mathbb Z}_{\geq 0}$, ordinification is injective on
${\mathcal W}_{\ell,d}^{-+}$ and on ${\mathcal W}_{\ell,d}^{+-}$.
\item[(II)] Ordinification maps a connected component ${\mathcal C}\subseteq 
{\mathcal W}_d^{-+}$ or ${\mathcal C}\subseteq {\mathcal W}_d^{+-}$ injectively into its image.
\end{itemize}
\end{lemma}
\begin{proof}
(I): Since $\operatorname{rev^c}({\bs n})$ 
is a sequence of blocks of $d-1$ entries of the form $(d,d-1,\ldots,j+1,j-1,\ldots,1) = \barindex{j}^c$; replacing each $
\barindex{j}^c$ with $\barindex{j}$ and reversing the order recovers ${\bs n}$ and therefore also ${\bs w}$. 
Therefore, ${\bs w}$ can be reconstructed uniquely from the data of ${\bs w^+}$ and $\ell({\bs n})$. 

(II): The operators $\worde_i,\wordf_i$ preserve the number of barred entries, hence
 $\ell({\bs n})$ is independent of ${\bs w}={\bs n}{\bs p}\in {\mathcal C}$ or ${\bs w}={\bs p}{\bs n}\in {\mathcal C}$. Now apply
(I).
\end{proof}

\begin{proposition} \label{prop:ordinification-isomorphism}
Fix a connected component ${\mathcal C}\subseteq {\mathcal W}_{\ell,d}^{-+}$ or 
${\mathcal C}\subseteq {\mathcal W}_{\ell,d}^{+-}$. The image of ${\mathcal C}$ 
under ordinification is
a connected component ${\mathcal C}'\subseteq {\mathcal W}_d$ and
${\bs w}\mapsto {\bs w}^+$ is a shifted crystal isomorphism
${\mathcal C}\to {\mathcal C}'$ with shift
$H=(\ell,\ell,\ldots,\ell)\in{\mathbb Z}^d$.
 \end{proposition}
 \begin{proof}
Pick ${\bs w}\in {\mathcal C}$ and let ${\mathcal C}'$ be the connected component of 
${\mathcal W}_d$ containing ${\bs w}^+$. By Lemma~\ref{lem-dual-words-crystals} and
the connectivity of ${\mathcal C}$, ${\bs x}^+\in {\mathcal C}'$ for every ${\bs x}\in {\mathcal C}$.
Hence we have a well-defined map $\psi:{\mathcal C}\to{\mathcal C}'$ by ${\bs x}\mapsto {\bs x}^+$. 
Now we apply Proposition~\ref{prop:inj-implies-bij} to $\psi$.
The injectivity hypothesis holds by
Lemma~\ref{lemma:May13aaa}(II). The first two requirements of \eqref{eq:shifted-crystal-morphism}
are checked by Lemma~\ref{lem-dual-words-crystals}. The weight function condition of
\eqref{eq:shifted-crystal-morphism} holds since
replacing each of the $\ell$ barred letters, say $\barindex{j}$, in
${\bs n}$ by $\barindex{j}^c$, increases the weight by $(1,\ldots,1)$.
Thus the shift is $H=(\ell,\ldots,\ell)$.
Proposition~\ref{prop:inj-implies-bij} now implies that $\psi$ is a
shifted crystal isomorphism onto ${\mathcal C}'$, as desired.
\end{proof}

We need two properties of Stembridge insertion from \cite{Stembridge}. 

\begin{theorem}[\!\!{\cite{Stembridge}}] \label{thm:Steminsprops}
\
\begin{itemize}
\item[(I)] If $i$ is a barred or unbarred label and $k\in {\mathbb Z}$ then 
\[\operatorname{shift}^k((U\div V) \leftarrow i)=\operatorname{shift}^k(U\div V)\leftarrow i.\]
\item[(II)] If ${\bs n}$ is a sequence of barred labels and ${\bs p}$ is a sequence of unbarred labels, then
\[(\emptyset\leftarrow {\bs n}) \leftarrow {\bs p} = (\emptyset\leftarrow {\bs p})\leftarrow {\bs n}.\]
\end{itemize}
\end{theorem}
\begin{proof}
If $k=0$ then (I) is trivial. If $k=1$ then (I) follows from the proof of
\cite[Theorem 3.6]{Stembridge}, specifically its first displayed equation.
The case $k=-1$ follows by applying the $k=1$ identity to
$\operatorname{shift}^{-1}(U\div V)$ and then applying
$\operatorname{shift}^{-1}$ to both sides. All other values of $k$
follow by iteration. Part (II) is \cite[Theorem 5.5]{Stembridge}.
\end{proof}

\begin{lemma} \label{lem:duality-shift-equiv}
For all ${\bs w}={\bs n}{\bs p}\in{\mathcal W}_d^{-+}$,
\begin{equation}\label{eqn:Apr30xyz}
\ins_r({\bs w}) =
\operatorname{shift}^{-\ell({\bs n})}
(\ins_r({\bs w}^+)).
\end{equation}
\end{lemma}

\begin{proof}
A straightforward check shows that
\begin{equation} \label{eq:single-insertion}
\emptyset\leftarrow\barindex{i} = \operatorname{shift}^{-1}(\emptyset\leftarrow \barindex{i}^c) = \varnothing \div \vcenter{\hbox{\begin{ytableau} i \end{ytableau}}}.
\end{equation}

To prove \eqref{eqn:Apr30xyz}, we induct on $k := \ell({\bs n})$. If $k=0$, then ${\bs n} = \emptyset$, and both sides of \eqref{eqn:Apr30xyz} equal $\emptyset\leftarrow {\bs p}$, so the base case holds. If $k>0$, write ${\bs n} = \barindex{i}{\bs n'}$. Now,
\begin{eqnarray}\nonumber
\emptyset\leftarrow{\bs n} = (\emptyset \leftarrow \barindex{i}) \leftarrow {\bs n'} 
& \stackrel{\eqref{eq:single-insertion}}{=} & \left(\operatorname{shift}^{-1}(\emptyset\leftarrow \barindex{i}^c)\right) \leftarrow {\bs n'}\\ \nonumber
&  \stackrel{\text{Theorem~\!\ref{thm:Steminsprops}(I)}}{=} & \operatorname{shift}^{-1}(\emptyset\leftarrow 
\barindex{i}^c \leftarrow {\bs n'})\\ \nonumber
 & \stackrel{\text{Theorem~\!\ref{thm:Steminsprops}(II)}}{=} & \operatorname{shift}^{-1}(\emptyset \leftarrow {\bs n'}\leftarrow \barindex{i}^c)\\ \nonumber
 & \stackrel{\text{ind.~hyp.}}{=} & \operatorname{shift}^{-1}(\operatorname{shift}^{-\ell({\bs n'})}
(\emptyset \leftarrow \operatorname{rev^c}({\bs n'}) \leftarrow \barindex{i}^c))\\ \nonumber
 &  \stackrel{\text{Theorem~\!\ref{thm:Steminsprops}(I)}}{=} & \operatorname{shift}^{-1-\ell({\bs n'})}
(\emptyset \leftarrow \operatorname{rev^c}({\bs n'})\leftarrow \barindex{i}^c)\\ \nonumber
& = & \operatorname{shift}^{-\ell({\bs n})}
(\emptyset \leftarrow \operatorname{rev^c}({\bs n})).\nonumber
\end{eqnarray}
Hence 
\[\!\emptyset\leftarrow{\bs n}{\bs p}= (\emptyset\leftarrow {\bs n})\leftarrow {\bs p}
=\operatorname{shift}^{-\ell({\bs n})}
(\emptyset \leftarrow \operatorname{rev^c}({\bs n})) \leftarrow{\bs p}
\stackrel{\text{Theorem~\!\ref{thm:Steminsprops}(I)}}{=}\operatorname{shift}^{-\ell({\bs n})}
(\ins_r({\bs w}^+)).\!\!\qedhere
\]
\end{proof}

\begin{lemma} \label{lem:col-ins-row-ins}
For all $\bs{w}={\bs p}{\bs n}\in {\mathcal W}_{d}^{+-}$,
\[
\ins_c(\bs{w}) = \operatorname{shift}^{-\ell(\bs{n})}(\ins_r(\bs{w}^+)).
\]
\end{lemma}

\begin{proof}
By \cite[Lemma~6.3]{Stroomer}, 
\begin{equation} \label{eq:row-col-insertion-relation}
i\rightarrow (U\div V) = \operatorname{shift}((U\div V)\leftarrow i^c).
\end{equation}
Ordinary and dual insertion are related by the swap map $\omega(U\div V) = V\div U$ via
\[
U\div V\leftarrow \barindex{i} = \omega(\omega(U\div V)\leftarrow i),
\qquad \barindex{i}\rightarrow U\div V = \omega(i\rightarrow \omega(U\div V)),
\]
and applying this to~\eqref{eq:row-col-insertion-relation} gives
\[
\barindex{i}\rightarrow (U\div V) = \operatorname{shift}^{-1}((U\div V)\leftarrow \barindex{i}^c).
\]
Iterating the combined use of these two facts and Theorem~\ref{thm:Steminsprops}(I) gives
\begin{equation}
\label{eqn:June21aaa}
w_1 w_2 \cdots w_\ell \rightarrow (U\div V) = \operatorname{shift}^k((U\div V)\leftarrow w_\ell^c\cdots w_1^c),
\end{equation}
where $k=\#\{i:w_i\in [d]\}-\#\{j: w_j \in [\barindex{d}]\}$ (cf.~\cite[Lemma~6.4]{Stroomer}). 

One checks that
$\emptyset\leftarrow i = \operatorname{shift}(\emptyset\leftarrow {i}^c) = \vcenter{\hbox{\begin{ytableau} i \end{ytableau}}}
\div\varnothing$ (cf.~\eqref{eq:single-insertion}), and using a similar argument as in the sequence of equalities in
Lemma~\ref{lem:duality-shift-equiv}, one derives
\begin{equation}
\label{eqn:June21ggg}
\emptyset\leftarrow {\bf p} = \operatorname{shift}^{\ell({\bs p})}(\emptyset\leftarrow \operatorname{rev}^c({\bs p})).
\end{equation}

Now we have
\begin{eqnarray}\nonumber
\bs{p}\bs{n}\rightarrow \varnothing & \stackrel{\eqref{eqn:June21aaa}}{=} &  \operatorname{shift}^{\ell(\bs{p})-\ell(\bs{n})}(\emptyset\leftarrow \operatorname{rev^c}(\bs{n}) \operatorname{rev^c}(\bs{p}))  \\ \nonumber
& \stackrel{\text{Theorem~\ref{thm:Steminsprops}(II)}}{=} & \operatorname{shift}^{\ell(\bs{p})-\ell(\bs{n})}(\emptyset\leftarrow \operatorname{rev^c}(\bs{p}) \operatorname{rev^c}(\bs{n}))\\ \nonumber
& \stackrel{\text{\eqref{eqn:June21ggg}, Theorem~\ref{thm:Steminsprops}(I)}}{=} & \operatorname{shift}^{-\ell(\bs{n})}(\emptyset\leftarrow \bs{p} \operatorname{rev^c}(\bs{n})). \nonumber\end{eqnarray}
Since ${\bs w}^+= {\bs p}\operatorname{rev^c}(\bs{n})\in {\mathcal W}_d$, $\varnothing\leftarrow {\bs w}^+=\operatorname{ins}_r({\bs w}^+)$,
and the lemma follows. \end{proof}

\subsection{Proof of Theorem~\ref{thm:stembridge-word-crystals}}\label{sec:proof32}
(I) Let ${\bs w}={\bs n}{\bs p}\in {\mathcal W}_d^{-+}$. Consider the map composition 
\[
{\bs w} \mapsto {\bs w}^+ \mapsto \ins_r({\bs w}^+) \mapsto \operatorname{shift}^{-\ell({\bs n})}(\ins_r({\bs w}^+)).
\]
By Proposition~\ref{prop:ordinification-isomorphism}, the first map is a shifted crystal isomorphism from 
$\mathcal{C}$ to a  connected component $\mathcal{C}'$ of ${\mathcal W}_d$,
with shift $(\ell,\ldots,\ell)$ where $\ell=\ell({\bs n})$.
RSK insertion, the second map, provides the isomorphism from 
$\mathcal{C}'$ to ${\mathcal B}_{\lambda}$, for some straight shape $\lambda$. 
Proposition~\ref{prop:tableau-shift-isomorphism} shows that 
the third map is a shifted crystal isomorphism from ${\mathcal B}_{\lambda}$ to
${\mathcal B}_{\operatorname{shift}^{-\ell({\bs n})}(\lambda)}$, with shift 
$-(\ell,\ldots,\ell)$. So, their composition is a shifted crystal isomorphism 
with shift $(\ell,\ldots,\ell) - (\ell,\ldots,\ell) = 0$, i.e., a crystal isomorphism. 
Finally,  by Lemma~\ref{lem:duality-shift-equiv},
$\operatorname{shift}^{-\ell({\bs n})}(\ins_r({\bs w}^+)) 
= \ins_r({\bs w})$.
Hence, in fact $\operatorname{shift}^{-\ell({\bs n})}(\lambda)=\gamma$, as desired.

(II): The proof is the same for $\bs{w}\in {\mathcal W}_d^{+-}$, 
except we apply Lemma~\ref{lem:col-ins-row-ins}.\qed

\subsection{Crystals of rational tableaux for profiles} \label{subsec:Cryprofiles}

Consider crystals $\mathcal{B}_1,\ldots,\mathcal{B}_k$ corresponding to weight lattices $\Lambda_1,\ldots,\Lambda_k$ and root systems $\Phi_1,\ldots,\Phi_k$ 
of ranks $a_1,\ldots,a_k$ and with index sets $I_1,\ldots,I_k$. Let 
\[\mathcal{B} = \mathcal{B}_1\times \cdots\times \mathcal{B}_k.\] 
For $1\le j\le k$ and $1\le i\le a_j$, define operators
$e_i^{(j)},f_i^{(j)}$ that act on $\mathcal{B}$ via the $j$th factor; i.e., if
$t_i^{(j)}$ is one of these operators, then
\[
t_i^{(j)} (b_1,\ldots, b_{j-1}, b_j, b_{j+1}, \ldots, b_k) := (b_1,\ldots, b_{j-1}, t_i(b_j), b_{j+1}, \ldots, b_k).
\]
Define the weight function
\begin{equation}
\label{eqn:July18aaa}
\operatorname{wt}_{{\mathcal B}}(b_1,\ldots,b_k)=(\operatorname{wt}_{{\mathcal B}_j}(b_j))_{j=1}^k.
\end{equation}
This datum makes $\mathcal{B}$ a crystal for the weight lattice $\Lambda_1\times \cdots\times \Lambda_k$ and root system $\Phi_1\sqcup \cdots\sqcup \Phi_k$; it is the \emph{Cartesian product crystal}. When $\mathcal{B}_j$ is the crystal of a representation $V_j$ of the group $G_j$, then $\mathcal{B}$ is the crystal for the exterior tensor product representation $V_1\boxtimes \cdots \boxtimes V_k$ of the direct product $G_1\times\cdots\times G_k$.

If ${\bf \Lambda} = (\gamma_1, \gamma_2,\ldots,\gamma_n)$ is a sequence of rational shapes, the Cartesian product crystal
\[
\mathcal{B}_{\bf \Lambda} := \mathcal{B}^{(d_1)}_{\gamma_1}\times\cdots\times \mathcal{B}^{(d_n)}_{\gamma_n},
\]
is a $GL(\d)$-crystal for the rational representation with highest weight ${\bf \Lambda}$. The elements of this crystal are the sequences of rational tableaux $(P_1,P_2,\ldots,P_n)$, 
whose shapes are ${\bf \Lambda}$. $\mathcal{B}^{(d_j)}_{\gamma_j}$ has rank $d_j-1$, so $\mathcal{B}_{\bf \Lambda}$ has crystal operators $\tabe_i^{(j)}, \tabf_i^{(j)}, 1\leq j\leq n, 1\leq i < d_j$.

Similarly, if $\mathcal{C}_i$ is a connected component of 
${\mathcal W}_{d_i}^{\operatorname{sep}}$ then we form the Cartesian product crystal $\mathcal{C}_1\times\cdots\times\mathcal{C}_n$, and these are the connected components of the product crystal 
\[{\mathcal W}_{\d}^{\operatorname{sep}} 
:= {\mathcal W}_{d_1}^{\operatorname{sep}}\times\cdots\times {\mathcal W}_{d_n}^{\operatorname{sep}}.\] 

In light of Theorem~\ref{thm:stembridge-word-crystals}, if $\bs{w}\in {\mathcal W}^{\operatorname{sep}}$ define
\begin{equation}\label{eqn:July23aaa}
\ins(\bs{w}) := \begin{cases}
\ins_r(\bs{w}), &\bs{w}\in {\mathcal W}_d^{-+}, \\
\ins_c(\bs{w}), &\bs{w}\in {\mathcal W}_d^{+-}. \\
\end{cases}
\end{equation}
This definition is potentially ambiguous if $\bs{w}\in \mathcal{W}_d\cup\mathcal{W}_d^-$; however, in this case $\ins_r(\bs{w}) = \ins_c(\bs{w})$ (see e.g.~\cite[\S A.2]{Fulton-book} for the case where $\bs{w}\in \mathcal{W}_d$; the case where $\bs{w}\in \mathcal{W}_d^-$ is analogous).

The \emph{profile} of $({\bs w}^{(1)}, \ldots, {\bs w}^{(n)})\in
{\mathcal W}_{\d}^{\operatorname{sep}}$ is the sequence of
rational shapes ${\bf \Lambda} = (\gamma_1, \gamma_2,\ldots,\gamma_n)$, where $\gamma_i$ is the shape of $\ins({\bs w}^{(i)})$.

Theorem~\ref{thm:stembridge-word-crystals} immediately gives:
\begin{corollary} \label{cor:cart-prod-crystals} \ 
The map
\[
\ins: ({\bs w}^{(1)}, \ldots, {\bs w}^{(n)}) \mapsto (\ins({\bs w}^{(1)}), \ldots,\ins({\bs w}^{(n)}))
\]
on $\mathcal{W}_{\d}^{\operatorname{sep}}$ restricts to an isomorphism 
$\mathcal{C}_1\times\cdots\times\mathcal{C}_n\to \mathcal{B}_{\bf \Lambda}$, 
where ${\bf \Lambda} $ is the profile of $({\bs w}^{(1)}, \ldots, {\bs w}^{(n)})$.
\end{corollary}

\subsection{Crystals and Knuth equivalence}
Suppose that ${\bs w},{\bs w}' \in {\mathcal W}_{\ell,d}^{-+}$ or ${\bs w},{\bs w}' \in {\mathcal W}_{\ell,d}^{+-}$. We say that 
${\bs w}$ and ${\bs w}'$ are \emph{Knuth equivalent} and write $\bs{w}\sim\bs{w}'$ if $\ins({\bs w}) = \ins({\bs w}')$. 
If ${\bs w},{\bs w}'\in {\mathcal W}_d$ this recovers the usual Knuth equivalence on words.
We want a crystal structure on Knuth equivalence classes. Let 
$\mathcal{K}_d^{\operatorname{sep}}=\{[\bs{w}]\}$ 
be the set of Knuth equivalence classes in $\mathcal{W}_d^{\operatorname{sep}}$, and 
\[\mathcal{K}_{\d}^{\operatorname{sep}} := \mathcal{K}_{d_1}^{\operatorname{sep}}\times\cdots\times 
\mathcal{K}_{d_n}^{\operatorname{sep}}=\{[{\bs w}^{(1)}, \ldots, {\bs w}^{(n)}]:=
([{\bs w}^{(1)}], \ldots, [{\bs w}^{(n)}])\}\]
be the set of Knuth equivalence classes in 
$\mathcal{W}_{\d}^{\operatorname{sep}}$. 

By definition, the insertion map and weight function on ${\mathcal W}_d^{\operatorname{sep}}$ and 
${\mathcal W}_{\bf d}^{\operatorname{sep}}$
descend to well-defined functions on ${\mathcal K}_d^{\operatorname{sep}}$ and 
${\mathcal K}_{\bf d}^{\operatorname{sep}}$. 
Respectively, define $\{\worde_i, \wordf_i\}$ and
$\{\worde_i^{(j)}, \wordf_i^{(j)}\}$ on $\mathcal{K}_d^{\operatorname{sep}}$ and 
$\mathcal{K}_{\d}^{\operatorname{sep}}$ by their action on a
choice of representative. Analogously define ${\mathcal K}_d^{-+}, {\mathcal K}_d^{+-}$, etc.

\begin{proposition}\label{prop:July24abc}
The crystal operators on $\mathcal{K}_d^{\operatorname{sep}}$ and $\mathcal{K}_{\d}^{\operatorname{sep}}$ are well-defined and satisfy Definition~\ref{def:crystal}. The insertion map
\[
\ins: ([{\bs w}^{(1)}, \ldots, {\bs w}^{(n)}]) \to (\ins([{\bs w}^{(1)}]), \ldots,
\ins([{\bs w}^{(n)}]))
\]
on $\mathcal{K}_{\d}^{\operatorname{sep}}$ is a crystal isomorphism on connected components $\mathcal{C}_1\times\cdots\times\mathcal{C}_n\to \mathcal{B}_{\bf \Lambda}$, where ${\bf \Lambda} $ is the profile of $({\bs w}^{(1)}, \ldots, {\bs w}^{(n)})$.
\end{proposition}

\begin{proof}
We first establish well-definedness of the crystal operators on ${\mathcal K}_d^{\operatorname{sep}}$. 
By Theorem~\ref{thm:stembridge-word-crystals}, $\operatorname{ins}$
is a crystal isomorphism on each component of ${\mathcal W}_d^{-+}$ and
${\mathcal W}_d^{+-}$. Now, suppose ${\bs w}\sim {\bs w}'$; 
by definition $\operatorname{ins}({\bs w})=\operatorname{ins}({\bs w}')$. Combining these facts,
\begin{equation}
\label{eqn:July23bbb}
\operatorname{ins}(\worde_i({\bs w}))=\tabe_i(\operatorname{ins}({\bs w}))=\tabe_i(\operatorname{ins}({\bs w}'))
=\operatorname{ins}(\worde_i({\bs w}'))\implies \worde_i({\bs w})\sim \worde_i({\bs w}'),
\end{equation}
as desired, where we use the usual convention for $\o$. Similarly,
$\wordf_i({\bs w})\sim \wordf_i({\bs w}')$.

Now let ${\mathcal C}$ be a connected component of ${\mathcal K}_d^{-+}$ (the same argument will apply for
${\mathcal K}_d^{+-}$). Since $\{\worde_i,\wordf_i\}$ preserve the number
of barred letters, ${\mathcal C}\subseteq {\mathcal K}_{\ell,d}^{-+}$ for some $\ell$. Let $\gamma$ be the shape
of $\operatorname{ins}({\bs w})$ for some $[{\bs w}]\in {\mathcal C}$. By \eqref{eqn:July23bbb}, and its $\wordf_i$-analogue,
one sees this shape is independent of the
choice of $[{\bs w}]\in {\mathcal C}$, and one has a well-defined map
\[\operatorname{ins}:{\mathcal C}\longrightarrow {\mathcal B}_\gamma\]
commuting with the crystal operators and preserving weight.
This map is injective: if
$\operatorname{ins}([{\bs w}])=\operatorname{ins}([{\bs w}'])$, then
${\bs w}\sim{\bs w}'$ by definition, and hence
$[{\bs w}]=[{\bs w}']$. We now apply 
Proposition~\ref{prop:inj-implies-bij} to conclude that ${\mathcal C}$ is a crystal and that
$\operatorname{ins}:{\mathcal C}\to{\mathcal B}_\gamma$ is a crystal isomorphism.
The result for ${\mathcal K}_{\bf d}^{\operatorname{sep}}$ follows componentwise.
\end{proof}
\section{The quiver crystal graph}\label{sec:quivercrystalgraph}
\subsection{Words of a matrix}\label{sec:S4}

An integer matrix $M$ has a \emph{row word} 
${\operatorname{row}}(M)$ obtained by reading its columns top to bottom and left to right recording $M_{ij}$ many $i$'s. The \emph{barred column word} $\overline{{\operatorname{col}}}(M)$
is obtained by reading the rows in English order and recording $M_{ij}$ many $\overline j$'s. 

If $M$ is associated to edge $e$ of $Q$, its \emph{$i$-reading word} for an incident vertex $i$,
denoted ${\sf{read}}_i(M)$,
 is $\operatorname{row}(M)$ if $i$ is a tail of $e$ and 
$\overline{\operatorname{col}}(M)$ if it is the head of $e$. Define ${\sf{read}}_Q({\bf M})$
to be the $n$-length list whose $i$-th entry is ${\sf read}_i(M_{i-1}){\sf read}_i(M_i)$
if $i$ is not a sink and ${\sf read}_i(M_{i}){\sf read}_i(M_{i-1})$ otherwise. By construction,
the $i$-th element of the reading word $\operatorname{read}_Q({\bf M})$
from \S~\ref{sec:mainresult} is the ordinification of the $i$-th element of ${\sf read}_Q({\bf M})$.

\subsection{The quiver crystal operators} \label{sec:ordbimoves}

We now define raising operators $\mate_i^{(j)}$ and lowering operators $\matf_i^{(j)}$ for each vertex $j$ of $Q$.
These act on ${\bf M}=(M_1,\ldots,M_{n-1})$ as follows.

First compute
\[
{\sf read}_Q({\bf M})
=
({\bs w}^{(1)},\ldots,{\bs w}^{(n)})
\in {\mathcal W}_{\bf d}^{\operatorname{sep}}.
\]
Define the weight of ${\bf M}$ by
\begin{equation}\label{eqn:Admwtdef}
\operatorname{wt}({\bf M}):=\operatorname{wt}\bigl({\sf read}_Q({\bf M})\bigr).
\end{equation}
\S~\ref{subsec:CryStembridge} defined operators
$\worde_i^{(j)}$ and $\wordf_i^{(j)}$ on ${\mathcal W}_{\bf d}^{\operatorname{sep}}$.
The affected entry of ${\bs w}^{(j)}$ determines a unique position $(r,c)$ in either
$M_{j-1}$ or $M_j$; denote this matrix by $M$.

If $\worde_i^{(j)}$ changes $i+1\to i$, define $\mate_i^{(j)}({\bf M})$
by
\[
M_{r,c}\mapsto M_{r,c}-1,
\qquad
M_{r-1,c}\mapsto M_{r-1,c}+1.
\]
If instead it changes $\barindex{i}\to\barindex{i+1}$, define
$\mate_i^{(j)}({\bf M})$ by
\[
M_{r,c}\mapsto M_{r,c}-1,
\qquad
M_{r,c+1}\mapsto M_{r,c+1}+1.
\]
If $\worde_i^{(j)}({\bs w}^{(j)})=\o$, set
$\mate_i^{(j)}({\bf M})=\o$.

Similarly, if $\wordf_i^{(j)}$ changes $i\to i+1$, define
$\matf_i^{(j)}({\bf M})$ by
\[
M_{r,c}\mapsto M_{r,c}-1,
\qquad
M_{r+1,c}\mapsto M_{r+1,c}+1.
\]
If instead it changes $\barindex{i+1}\to\barindex{i}$, define
$\matf_i^{(j)}({\bf M})$ by
\[
M_{r,c}\mapsto M_{r,c}-1,
\qquad
M_{r,c-1}\mapsto M_{r,c-1}+1.
\]
If $\wordf_i^{(j)}({\bs w}^{(j)})=\o$, set
$\matf_i^{(j)}({\bf M})=\o$.

In the type $A_2$ case, these are the \emph{bicrystal operators} $e_i^{\operatorname{row}}, f_i^{\operatorname{row}},
e_i^{\operatorname{col}}, f_i^{\operatorname{col}}$ of
\cite{Danilov, van} (we follow the notation of \cite{AAA}).

\begin{theorem}\label{thm:S6.3}
If ${\bf M}\in\Std(\Omega)$ and
$\phi\in\{\mate_i^{(j)},\matf_i^{(j)}\}$, then $\phi({\bf M})=\o$ or $\phi({\bf M})\in \operatorname{Adm}(\Omega)$.
\end{theorem}
\begin{proof}
Call an antidiagonal of ${\mathcal R}_{[p,q]}$ of length greater than $r_{[p,q]}$
\emph{long}; we read antidiagonals from northeast to southwest. Since ${\bf M}$ and
$\phi({\bf M})$ differ only in that a single unit has moved from a box $a$ to an adjacent
box $a'$, if $\phi({\bf M})$ has a long antidiagonal $D$ then $a'\in D$ and
${\bf M}_{a'}=0$; otherwise $D$ is a long antidiagonal for ${\bf M}$. So in each case below it suffices to
modify $D$ into an antidiagonal of ${\bf M}$ in ${\mathcal R}_{[p,q]}$ of at least the same
length.

There are four local orientations at $j$. If $j$ is a source or a sink then no phantom
region is incident to $j$: here $M_{j-1}$ and $M_j$ occupy a single rectangle $N$ of the
$Q$-shape whose row word (respectively, barred column word) is ${\bs w}^{(j)}$, and $\phi$
moves a unit along a column (respectively, a row) of $N$. Subcase~1.1 and the first two
paragraphs of Subcase~1.2 below apply verbatim with $N$ in place of $M_j$ and ${\bs w}^{(j)}$ in place of the modified reading word displayed there, no phantom region being present;
for a sink one first transposes, as in Subcases~1.3 and~1.4. This
covers $j=1,n$, where $N$ is a single block. We may therefore assume that a phantom region
is incident to $j$, which leaves two cases.

\noindent\emph{Case 1:} $\left(
\begin{tikzpicture}[baseline=-0.5ex]
\node[circle,fill,inner sep=1pt,label=below:{\scriptsize $j-1$}] (a) at (0,0) {};
\node[circle,fill,inner sep=1pt,label=below:{\scriptsize $j$}]   (b) at (.9,0) {};
\node[circle,fill,inner sep=1pt,label=below:{\scriptsize $j+1$}] (c) at (1.8,0) {};
\draw[->] (a) -- node[above,font=\scriptsize] {$M_{j-1}$} (b);
\draw[->] (b) -- node[above,font=\scriptsize] {$M_j$} (c);
\end{tikzpicture}\right)$ Hence locally, the $Q$-shape of ${\bf M}$ looks like:
$\vcenter
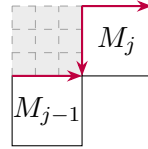
{\hbox{\begin{tikzpicture}[scale=1.1, baseline=-.5ex]
\draw (0,0) rectangle (.84,.84);
\node at (.42,.42) {\small{$M_{j-1}$}};
{\phantomtile{0}{.84}{3}{3}{}}
\draw (.84,.84) rectangle (1.68,1.68);
\node at (1.26,1.26) {\small{$M_j$}};
\draw[purple, -{Stealth[length=2mm,width=1.5mm]}, thick] (0,.84) -- (.84,.84);
\draw[purple, -{Stealth[length=2mm,width=1.5mm]}, thick] (.84,1.68) -- (.84,.84);
\draw[purple, -{Stealth[length=2mm,width=1.5mm]}, thick] (.84,1.68) -- (1.68,1.68);
\end{tikzpicture}}}$

\smallskip
\noindent\emph{Subcase 1.1:} ($\matf_i^{{(j)}}$ moves a value of $1$ down in $M_j$)
Suppose the value $1$ moves down from row $i$ in box $a=(i,c)$ of $M_j$ to box
$a'=(i+1,c)$, but ${\bf M}'=\matf_i^{(j)}({\bf M})$ contains a long antidiagonal $D$. The
box $d$ of $D$ immediately before $a'$ must lie in row $i$, and hence in $M_j$: otherwise
$d$ lies above row $i$, or $a'$ is the first box of $D$, and replacing $a'$ by $a$ gives a
long antidiagonal for ${\bf M}$. There are two possibilities. The first is that an unpaired
{\tt )} in the modified reading word
$\operatorname{rev}_{-}(\overline{\operatorname{col}}(M_{j-1})\operatorname{row}(M_j))$ is
associated with $d$. However, this means the ``$1$'' in $a$ does not correspond to the
rightmost unpaired {\tt )}, and hence $\matf_{i}^{(j)}$ does not move this $1$, a
contradiction. The second is that every {\tt )} associated to $d$ is paired. This means
there is a nonzero entry in a box $b$ in row $i+1$ that supplies the paired {\tt (}; the
assumption that the {\tt )} associated to $a$ is unpaired says that $b$ is in a column
strictly between those of $a$ and $d$. Then replacing $a'$ with $b$ gives a long
antidiagonal in ${\bf M}$, the final contradiction.

\smallskip
\noindent\emph{Subcase 1.2:} ($\mate_i^{(j)}$ moves a value of $1$ up in $M_j$)
We refer to the following diagram, where I, II and III denote the parts of the $Q$-shape
occupied by $M_{j-1}$, the phantom region and $M_j$:
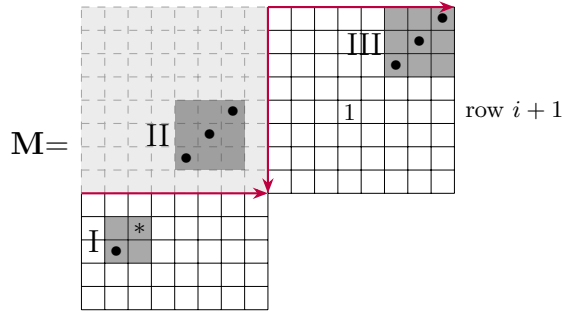
\begin{figure}[h]
  \centering
  \begin{tikzpicture}[scale=1.1, baseline=0ex]
\tile{0}{0}{8}{5}{
  \fill[gray!150, opacity=.45] (.28,.28+.28) rectangle (.84,.28+.84);
  \numcell{2}{-1}{*}
  \numcell{1}{0}{\bullet}
}
{\phantomtile{0}{1.4}{8}{8}{
 \fill[gray!150, opacity=.45] (.84+.28,.28) rectangle (0.84+.56+.28+.28,.56+.28+.28);
\numcell{4}{1}{\bullet}
\numcell{5}{0}{\bullet}
\numcell{6}{-1}{\bullet}
}}
\tile{2.24}{1.4}{8}{8}{
\fill[gray!150, opacity=.45] (.84+.28+.28,.28+.84+.28) rectangle (0.84+.56+.28+.28+.28,.56+.28+.28+.84+.28);
  \numcell{3}{-1}{1}
  \numcell{5}{-3}{\bullet}
  \numcell{6}{-4}{\bullet}
  \numcell{7}{-5}{\bullet}
}
\draw[purple, -{Stealth[length=2mm,width=1.5mm]}, thick] (0,1.4) -- (2.24,1.4);
\draw[purple, -{Stealth[length=2mm,width=1.5mm]}, thick] (2.24,1.4+2.24) -- (2.24,1.4);
\draw[purple, -{Stealth[length=2mm,width=1.5mm]}, thick] (2.24,1.4+2.24) -- (4.48,1.4+2.24);
\node at (0.15,0.82) {I};
\node at (.92,2.1){II};
\node at (3.4,3.2){III};
\node at (5.2,2.4){\scriptsize row $i+1$};
\node at (-.5,2){{\bf M}=};
\end{tikzpicture}
\end{figure}

Suppose the value $1$ moves up from $a=(i+1,c)$ to $a'=(i,c)$, and let $D$ be a long
antidiagonal of ${\bf M}'=\mate_i^{(j)}({\bf M})$ in ${\mathcal R}_{[p,q]}$.

First suppose the box of $D$ immediately after $a'$ is a box $b=(i+1,c_1)$ of III; then
$c_1<c$ and ${\bf M}_b>0$. In the modified reading word, $b$ contributes a {\tt (}
preceding the {\tt (} of $a$. As the latter is the leftmost unpaired {\tt (}, the former is
paired with a {\tt )} strictly between them; both {\tt (}'s lie in the segment
$\operatorname{row}(M_j)$, so this {\tt )} comes from a positive entry $g=(i,c_2)$ of $M_j$
with $c_1<c_2<c$. Replacing $a'$ by $g$ in $D$ gives a long antidiagonal for ${\bf M}$, a
contradiction. Consequently, if the box of $D$ after $a'$ does not lie in II (in
particular, if there is no such box), then replacing $a'$ by $a$ gives a long antidiagonal
for ${\bf M}$. So assume it does; then ${\mathcal R}_{[p,q]}$ contains II and $a'$ is the
last box of $D$ in III.

Let $\ast$ be the first box of $D$ in I and $k$ its column, with $k=0$ if there is no such
box. Since the boxes of $D$ in II lie in rows $i+1,\ldots,d_j$ and in columns
$k+1,\ldots,d_j$, there are at most $d_j-\max(i,k)$ of them; conversely, the antidiagonal
of the southeast $(d_j-\max(i,k))$-square of II is compatible with every box of $D$ outside
II, as is straightforward to check. The same holds after $a'$ is deleted from $D$ or
replaced by a box of III in a row $\le i$. We use this repeatedly below.

If $k\le i-1$, delete $a'$ from $D$: the last box of $D$ in III now lies in a row $<i$, so
II may contribute $d_j-i+1$ boxes, and we obtain an antidiagonal for ${\bf M}$ of length at
least $|D|$, a contradiction. If $k\ge i+1$, replace $a'$ by $a$: this changes neither
$\max(i,k)$ nor the length, a contradiction again.

Finally, suppose $k=i$. In the modified reading word, the $1$ in row $i+1$ corresponds to
the leftmost unpaired {\tt (}. However, box $\ast$ in column $i$ also corresponds to a {\tt (}.
If that latter {\tt (} is unpaired, then in fact the crystal move affects $M_{j-1}$, not
$M_j$ as assumed. So it is paired, and the only two ways for this are:
\begin{itemize}
\item[(1)] a positive entry in box $z$ in column $i+1$ of $M_{j-1}$, northeast of $\ast$. Replace
$a'$ by $a$ and insert $z$ immediately before $\ast$: the first box in I is now $z$, so II
contributes one box fewer, while $D$ has gained $z$.
\item[(2)] a positive entry in box $g=(i,c_2)$ of $M_j$ with $c_2<c$. Replace $a'$ by $g$, which
changes neither $\max(i,k)$ nor the length.
\end{itemize}
Either way we obtain a long antidiagonal for ${\bf M}$, completing Subcase~1.2.

\smallskip
Transposing the $Q$-shape, or rotating it by $180^{\circ}$, reverses the bracket sequence
of $\operatorname{rev}_{-}({\bs w}^{(j)})$ and interchanges {\tt (} with {\tt )}; hence
either operation exchanges the leftmost unpaired {\tt (} with the rightmost unpaired
{\tt )}, that is, $\mate_i^{(j)}$ with $\matf_i^{(j)}$, and both preserve antidiagonals.

\smallskip
\noindent\emph{Subcase 1.3:} ($\matf_i^{(j)}$ moves a value of $1$ left in $M_{j-1}$) This
is the transposed version of Subcase~1.2.

\smallskip
\noindent\emph{Subcase 1.4:} ($\mate_i^{(j)}$ moves a value of $1$ right in $M_{j-1}$) This
is the transposed version of Subcase~1.1.

\medskip
\noindent
\emph{Case 2:} $\left(
\begin{tikzpicture}[baseline=-0.5ex]
\node[circle,fill,inner sep=1pt,label=below:{\scriptsize $j-1$}] (a) at (0,0) {};
\node[circle,fill,inner sep=1pt,label=below:{\scriptsize $j$}]   (b) at (.9,0) {};
\node[circle,fill,inner sep=1pt,label=below:{\scriptsize $j+1$}] (c) at (1.8,0) {};
\draw[<-] (a) -- node[above,font=\scriptsize] {$M_{j-1}$} (b);
\draw[<-] (b) -- node[above,font=\scriptsize] {$M_j$} (c);
\end{tikzpicture}\right)$
Here, the $Q$-shape of ${\bf M}$ locally looks like:
$\vcenter
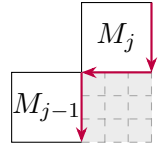
{\hbox{\begin{tikzpicture}[scale=1.1, baseline=-.5ex]
\draw (0,0) rectangle (.84,.84);
\node at (.42,.42) {\small $M_{j-1}$};
\draw (.84,.84) rectangle (1.68,1.68);
\node at (1.26,1.26) {\small{$M_j$}};
\phantomtile{.84}{0}{3}{3}{}
\draw[purple, -{Stealth[length=2mm,width=1.5mm]}, thick] (1.68,.84) -- (0.84,.84);
\draw[purple, -{Stealth[length=2mm,width=1.5mm]}, thick] (1.68,1.68) -- (1.68,.84);
\draw[purple, -{Stealth[length=2mm,width=1.5mm]}, thick] (.84,.84) -- (.84,0);
\end{tikzpicture}}}$

The arguments here are the $180$ degree rotation of those in Case~1. Namely,
$\mate^{(j)}_i$ either moves a value of $1$ up in $M_{j-1}$ or right in $M_j$, and one
applies the arguments of Subcases~1.1 and~1.3 respectively; meanwhile $\matf^{(j)}_i$
either moves a value of $1$ down in $M_{j-1}$ or left in $M_j$, and these are argued as in
Subcases~1.2 and~1.4.
\end{proof}

\subsection{From matrices to tableaux}
For each ${\bf M}\in \Std(\Omega)$, associate  rational tableaux 
\[\operatorname{quiverRSK}({\bf M}) := \ins({\sf read}_Q({\bf M})).\]  
Let ${\sf profile}({\bf M})$ be the sequence of rational shapes appearing in ${\sf read}_Q({\bf M})$. 
By Theorem~\ref{thm:S6.3} we may refer to a connected component $\mathcal{C}({\bf M})\subseteq
\Std(\Omega)$ of ${\bf M}\in \Std(\Omega)$. 

\begin{theorem} \label{prop:quiverRSK-crystal-isom}
Let $\mathcal{C}$ be a connected component of $\Std(\Omega)$. 
The map 
\[\operatorname{quiverRSK}\!|_\mathcal{C}: \mathcal{C}\to \mathcal{B}_{\bf \Lambda}\] 
is a crystal isomorphism where
${\bf \Lambda}={\sf profile}({\bf M})$ is independent of the choice of ${\bf M}\in {\mathcal C}$.
\end{theorem}

To prove Theorem~\ref{prop:quiverRSK-crystal-isom}, we again use Proposition~\ref{prop:inj-implies-bij}. This involves showing that the reading map commutes with quiver crystal operators up to Knuth equivalence, and that on any connected component, the reading map is injective even up to Knuth classes. First, we need:
\begin{lemma}\label{lem:separated-crystal-combined-word}
Let ${\bs w}\in {\mathcal W}_{d}^{\operatorname{sep}}$ be the concatenation ${\bs w} = {\bs a}{\bs b}$ such that ${\bs a},{\bs b}\in {\mathcal W}_{d}\cup {\mathcal W}_{d}^-$. Then 
$\worde_i({\bs w}) \in\left\{\worde_i({\bs a}){\bs b},\ {\bs a}\worde_i({\bs b})\right\}$, and similarly for $\wordf_i$.
(By convention, ${\bs a}\o=\o{\bs b}=\o$.)
\end{lemma}
\begin{proof}
We prove the lemma for $\worde_i({\bs w})$ (the proof for $\wordf_i({\bs w})$ is similar).
First, suppose $\bs{w}$ is an unbarred word, i.e., ${\bs a}, {\bs b}\in {\mathcal W}_{d}$. Let $\text{{\tt (}}_u$
be the leftmost unpaired {\tt (} in $\operatorname{parens}_i({\bs w})=\operatorname{parens}_i({\bs a})\operatorname{parens}_i({\bs b})$.
 If $\text{{\tt (}}_u$ does not exist then no unpaired {\tt (} exists in ${\bs b}$, so $\worde_i({\bs w}) = \o = {\bs a}\o={\bs a}\worde_i({\bs b})$. So assume $\text{{\tt (}}_u$ exists.
If $\text{{\tt (}}_u$ is in ${\bs b}$, it is also the leftmost unpaired {\tt (} in ${\bs b}$ and 
$e_i({\bs w}) = {\bs a}e_i({\bs b})$.
Finally, if $\text{{\tt (}}_u$ lies in ${\bs a}$ then $\text{{\tt (}}_u$ is also unpaired in the substring $\operatorname{parens}_i({\bs a})$. 
Take any $\text{{\tt (}}_p$ strictly left of $\text{{\tt (}}_u$. Since $\text{{\tt (}}_p$ is
paired in ${\bs w}$, its pairing {\tt )} in ${\bs w}$ is left of $\text{{\tt (}}_u$, and therefore also in ${\bs a}$. Therefore $\text{{\tt (}}_p$ is paired in $\operatorname{parens}_i({\bs a})$. Thus, $\worde_i({\bs w}) = \worde_i({\bs a}){\bs b}$.

In general, by the proof of Lemma~\ref{lem-dual-words-crystals},
$\operatorname{parens}_i({\bs w}^+)$ is obtained from
$\operatorname{parens}_i({\bs w})$ by inserting contiguous {\tt ()} pairs.
Thus, by the unbarred case above:
\begin{itemize}
\item If $\bs{a}\in {\mathcal W}_{d}^-, \bs{b}\in  {\mathcal W}_{d}$, then 
$
\worde_i({\bs w})^+ = \worde_i({\bs w}^+) =  \worde_i({\bs a}^+ {\bs b}) \in 
\left\{\worde_i({\bs a}^+){\bs b},\ {\bs a}^+\worde_i({\bs b})\right\}.$
\item If $\bs{a}\in {\mathcal W}_{d}, \bs{b}\in {\mathcal W}_{d}^-$, then 
$\worde_i({\bs w})^+ = \worde_i({\bs w}^+) = \worde_i({\bs a}{\bs b}^+) \in 
\left\{\worde_i({\bs a}){\bs b}^+, \ {\bs a}\worde_i({\bs b}^+)\right\}$.
\item If $\bs{a},\bs{b}\in {\mathcal W}_{d}^-$, then
$\worde_i({\bs w})^+ = \worde_i({\bs w}^+) = \worde_i({\bs b}^+{\bs a}^+) \in 
\left\{\worde_i({\bs b}^+){\bs a}^+, \ {\bs b}^+\worde_i({\bs a}^+)\right\}$.
\end{itemize}
By applying Lemma~\ref{lem-dual-words-crystals} to each alternative above, we 
have
\[\worde_i({\bs w})^+\in\left\{(\worde_i({\bs a}){\bs b})^+, \ ({\bs a}\worde_i({\bs b}))^+\right\}.\]
If $\worde_i({\bs w})=\o$, the claim follows immediately. Otherwise, the
matching non-null words have the same number of barred letters, so the claim
follows from Lemma~\ref{lemma:May13aaa}(I).
\end{proof}

\begin{proposition} \label{lem:matrix-quiver-crystal-ops-RSK}
${\sf read}_Q(-)$ commutes with the quiver crystal operators up to Knuth equivalence, that is,
\[\worde_i^{(j)}([{\sf read}_Q({\bf M})]) = [{\sf read}_Q(\mate_i^{(j)}({\bf M}))],\]
\[\wordf_i^{(j)}([{\sf read}_Q({\bf M})]) = [{\sf read}_Q(\matf_i^{(j)}({\bf M}))],\]
for all $1\leq j\leq n, 1\leq i< d_j$. Furthermore,
if ${\bf N}\in \mathcal{C}({\bf M})$ then ${\sf profile}({\bf N})={\sf profile}({\bf M})$.
\end{proposition}

\begin{proof} 
We prove the commutation claim for the raising operators; the proof for the lowering operators is the same.
Let ${\sf read}_Q({\bf M}) = ({\bs w}^{(1)},\ldots, {\bs w}^{(n)})$. By definition,
\[
\worde_i^{(j)}({\sf read}_Q({\bf M})) = ({\bs w}^{(1)},\ldots,  {\bs w}^{(j-1)}, \worde_i({\bs w}^{(j)}), {\bs w}^{(j+1)}, \ldots, {\bs w}^{(n)}).
\]
Let ${\sf read}_Q(\mate_i^{(j)}({\bf M})) = ({\bs v}^{(1)}, \ldots, {\bs v}^{(n)})$.
We need to show that
\begin{equation} \label{eq:ei-matrix-equals-word}
 {\bs v}^{(j)}\sim \worde_i({\bs w}^{(j)}), \qquad \text{ and for $k\ne j$, } {\bs v}^{(k)} \sim {\bs w}^{(k)}.
\end{equation}

Let ${\bf M}^{(j)}$ be the part of $Q$-shape of ${\bf M}$ corresponding to vertex $j$, meaning those rows and columns that
involve $M_{j-1}$ and $M_j$. There are four cases:

\smallskip
\noindent\emph{Case 1:}
$\left(
\begin{tikzpicture}[baseline=-0.5ex]
\node[circle,fill,inner sep=1pt,label=below:{\scriptsize $j-1$}] (a) at (0,0) {};
\node[circle,fill,inner sep=1pt,label=below:{\scriptsize $j$}]   (b) at (.9,0) {};
\node[circle,fill,inner sep=1pt,label=below:{\scriptsize $j+1$}] (c) at (1.8,0) {};

\draw[<-] (a) -- node[above,font=\scriptsize] {$M_{j-1}$} (b);
\draw[->] (b) -- node[above,font=\scriptsize] {$M_j$} (c);
\end{tikzpicture}\right)$
$\qquad {\bf M}^{(j)} =\vcenter{\hbox{\begin{tikzpicture}[scale=1.1, baseline=-.5ex]

\draw (0,0) rectangle (.84,.84);
\node at (.42,.42) {\small{$M_{j-1}$}};

\draw (.84,0) rectangle (1.68,.84);
\node at (1.26,0.42) {\small{$M_j$}};

\draw[purple, -{Stealth[length=2mm,width=1.5mm]}, thick] (.84,.84) -- (.84,0);
\draw[purple, -{Stealth[length=2mm,width=1.5mm]}, thick] (.84,.84) -- (1.68,.84);
\end{tikzpicture}}}$ $\qquad {\bs w}^{(j)} = \operatorname{row}(M_{j-1})\operatorname{row}(M_j)$

\noindent\emph{Case 2:}
$\left(
\begin{tikzpicture}[baseline=-0.5ex]
\node[circle,fill,inner sep=1pt,label=below:{\scriptsize $j-1$}] (a) at (0,0) {};
\node[circle,fill,inner sep=1pt,label=below:{\scriptsize $j$}]   (b) at (.9,0) {};
\node[circle,fill,inner sep=1pt,label=below:{\scriptsize $j+1$}] (c) at (1.8,0) {};

\draw[->] (a) -- node[above,font=\scriptsize] {$M_{j-1}$} (b);
\draw[<-] (b) -- node[above,font=\scriptsize] {$M_j$} (c);
\end{tikzpicture}\right)$
$\qquad {\bf M}^{(j)} =\vcenter{\hbox{\begin{tikzpicture}[scale=1.1, baseline=-.5ex]

\draw (0,0) rectangle (.84,.84);
\node at (.42,.42) {\small{$M_{j-1}$}};

\draw (0,.84) rectangle (.84,1.68);
\node at (0.42,1.26) {\small{$M_j$}};

\draw[purple, -{Stealth[length=2mm,width=1.5mm]}, thick] (0,.84) -- (.84,.84);
\draw[purple, -{Stealth[length=2mm,width=1.5mm]}, thick] (.84,.84) -- (.84,1.68);
\end{tikzpicture}}}$ $\qquad \qquad {\bs w}^{(j)} =\overline{\operatorname{col}}(M_j) \overline{\operatorname{col}}(M_{j-1})$

\noindent\emph{Case 3:}
$\left(
\begin{tikzpicture}[baseline=-0.5ex]
\node[circle,fill,inner sep=1pt,label=below:{\scriptsize $j-1$}] (a) at (0,0) {};
\node[circle,fill,inner sep=1pt,label=below:{\scriptsize $j$}]   (b) at (.9,0) {};
\node[circle,fill,inner sep=1pt,label=below:{\scriptsize $j+1$}] (c) at (1.8,0) {};

\draw[->] (a) -- node[above,font=\scriptsize] {$M_{j-1}$} (b);
\draw[->] (b) -- node[above,font=\scriptsize] {$M_j$} (c);
\end{tikzpicture}\right)$ $\qquad {\bf M}^{(j)}=\vcenter{\hbox{\begin{tikzpicture}[scale=1.1, baseline=-.5ex]

\draw (0,0) rectangle (.84,.84);
\node at (.42,.42) {\small{$M_{j-1}$}};

{\phantomtile{0}{.84}{3}{3}{
}}

\draw (.84,.84) rectangle (1.68,1.68);
\node at (1.26,1.26) {\small{$M_j$}};

\draw[purple, -{Stealth[length=2mm,width=1.5mm]}, thick] (0,.84) -- (.84,.84);
\draw[purple, -{Stealth[length=2mm,width=1.5mm]}, thick] (.84,1.68) -- (.84,.84);
\draw[purple, -{Stealth[length=2mm,width=1.5mm]}, thick] (.84,1.68) -- (1.68,1.68);
\end{tikzpicture}}}$ $\qquad {\bs w}^{(j)} = \overline{\operatorname{col}}(M_{j-1})\operatorname{row}(M_j)$

\noindent\emph{Case 4:}
$\left(
\begin{tikzpicture}[baseline=-0.5ex]
\node[circle,fill,inner sep=1pt,label=below:{\scriptsize $j-1$}] (a) at (0,0) {};
\node[circle,fill,inner sep=1pt,label=below:{\scriptsize $j$}]   (b) at (.9,0) {};
\node[circle,fill,inner sep=1pt,label=below:{\scriptsize $j+1$}] (c) at (1.8,0) {};

\draw[<-] (a) -- node[above,font=\scriptsize] {$M_{j-1}$} (b);
\draw[<-] (b) -- node[above,font=\scriptsize] {$M_j$} (c);
\end{tikzpicture}\right)$ 
$\qquad {\bf M}^{(j)}=\vcenter{\hbox{\begin{tikzpicture}[scale=1.1, baseline=-.5ex]

\draw (0,0) rectangle (.84,.84);
\node at (.42,.42) {\small $M_{j-1}$};

\draw (.84,.84) rectangle (1.68,1.68);
\node at (1.26,1.26) {\small{$M_j$}};

\phantomtile{.84}{0}{3}{3}{}

\draw[purple, -{Stealth[length=2mm,width=1.5mm]}, thick] (1.68,.84) -- (0.84,.84);
\draw[purple, -{Stealth[length=2mm,width=1.5mm]}, thick] (1.68,1.68) -- (1.68,.84);
\draw[purple, -{Stealth[length=2mm,width=1.5mm]}, thick] (.84,.84) -- (.84,0);
\end{tikzpicture}}}$
$\qquad {\bs w}^{(j)} = \operatorname{row}(M_{j-1})\overline{\operatorname{col}}(M_j)$

(The endpoint cases $j=1,n$ are the same, with the nonexistent adjacent block omitted.)

We will show that~\eqref{eq:ei-matrix-equals-word} holds in Case 3; the arguments for the other cases are similar.

If $|k-j|\geq 2$, ${\bs v}^{(k)} = {\bs w}^{(k)}$ since ${\bf M}^{(k)}$ does not overlap with ${\bf M}^{(j)}$.
Thus, it remains to check ${\bs v}^{(j\pm 1)}$ where we have
\[
{\bf M}^{(j-1)} =
\vcenter{\hbox{\begin{tikzpicture}[scale=1.1, baseline=-.5ex]

\draw (0,0) rectangle (.84,.84);
\node at (.42,.42) {\small{$M_{j-2}$}};

\draw (.84,0) rectangle (1.68,.84);
\node at (1.26,0.42) {\small{$M_{j-1}$}};

\draw[purple, -{Stealth[length=2mm,width=1.5mm]}, thick] (.84,.84) -- (.84,0);
\draw[purple, -{Stealth[length=2mm,width=1.5mm]}, thick] (.84,.84) -- (1.68,.84);
\end{tikzpicture}}}
\,\text{or}\,
\vcenter{\hbox{\begin{tikzpicture}[scale=1.1, baseline=-.5ex]

\draw (0,0) rectangle (.84,.84);
\node at (.42,.42) {\small{$M_{j-2}$}};

{\phantomtile{0}{.84}{3}{3}{
}}

\draw (.84,.84) rectangle (1.68,1.68);
\node at (1.26,1.26) {\small{$M_{j-1}$}};

\draw[purple, -{Stealth[length=2mm,width=1.5mm]}, thick] (0,.84) -- (.84,.84);
\draw[purple, -{Stealth[length=2mm,width=1.5mm]}, thick] (.84,1.68) -- (.84,.84);
\draw[purple, -{Stealth[length=2mm,width=1.5mm]}, thick] (.84,1.68) -- (1.68,1.68);
\end{tikzpicture}}},
\qquad
{\bf M}^{(j+1)} =\vcenter{\hbox{\begin{tikzpicture}[scale=1.1, baseline=-.5ex]

\draw (0,0) rectangle (.84,.84);
\node at (.42,.42) {\small{$M_{j}$}};

\draw (0,.84) rectangle (.84,1.68);
\node at (0.42,1.26) {\small{$M_{j+1}$}};

\draw[purple, -{Stealth[length=2mm,width=1.5mm]}, thick] (0,.84) -- (.84,.84);
\draw[purple, -{Stealth[length=2mm,width=1.5mm]}, thick] (.84,.84) -- (.84,1.68);
\end{tikzpicture}}}
\,\text{or} \,\, \vcenter{\hbox{\begin{tikzpicture}[scale=1.1, baseline=-.5ex]

\draw (0,0) rectangle (.84,.84);
\node at (.42,.42) {\small{$M_{j}$}};

{\phantomtile{0}{.84}{3}{3}{
}}

\draw (.84,.84) rectangle (1.68,1.68);
\node at (1.26,1.26) {\small{$M_{j+1}$}};

\draw[purple, -{Stealth[length=2mm,width=1.5mm]}, thick] (0,.84) -- (.84,.84);
\draw[purple, -{Stealth[length=2mm,width=1.5mm]}, thick] (.84,1.68) -- (.84,.84);
\draw[purple, -{Stealth[length=2mm,width=1.5mm]}, thick] (.84,1.68) -- (1.68,1.68);
\end{tikzpicture}}}.\]
The application of $\mate_i^{(j)}$ to ${\bf M}$ sends ${\bs w}^{(j)}$ to $\worde_i({\bs w}^{(j)})$, which by Lemma~\ref{lem:separated-crystal-combined-word} means that $\overline{\operatorname{col}}(M_{j-1})\operatorname{row}(M_j)$ is sent to $\worde_i(\overline{\operatorname{col}}(M_{j-1})\operatorname{row}(M_j)) = \worde_i(\overline{\operatorname{col}}(M_{j-1}))\operatorname{row}(M_j)$ or  otherwise
$\overline{\operatorname{col}}(M_{j-1})\worde_i(\operatorname{row}(M_j))$. In the former case,
\[
\mate_i^{(j)}({\bf M})^{(j-1)} =
\vcenter{\hbox{\begin{tikzpicture}[scale=1.1, baseline=-.5ex]

\draw (0,0) rectangle (.84,.84);
\node at (.42,.42) {\small{$M_{j-2}$}};

\draw (.84,0) rectangle (2.60,.84);
\node at (1.70,0.42) {\small{$e_i^{\mathrm{col}}(M_{j-1})$}};
\draw[purple, -{Stealth[length=2mm,width=1.5mm]}, thick] (.84,.84) -- (.84,0);
\draw[purple, -{Stealth[length=2mm,width=1.5mm]}, thick] (.84,.84) -- (2.60,.84);
\end{tikzpicture}}}
\, \, \text{or}\, \,
\vcenter{\hbox{\begin{tikzpicture}[scale=1.1, baseline=-.5ex]

\draw (0,0) rectangle (.84,.84);
\node at (.42,.42) {\small{$M_{j-2}$}};

{\phantomtile{0}{.84}{3}{3}{
}}

\draw (.84,.84) rectangle (2.60,1.68);
\node at (1.70,1.26) {\small{$e_i^{\mathrm{col}}(M_{j-1})$}};

\draw[purple, -{Stealth[length=2mm,width=1.5mm]}, thick] (0,.84) -- (.84,.84);
\draw[purple, -{Stealth[length=2mm,width=1.5mm]}, thick] (.84,1.68) -- (.84,.84);
\draw[purple, -{Stealth[length=2mm,width=1.5mm]}, thick] (.84,1.68) -- (2.60,1.68);
\end{tikzpicture}}},
\qquad
\mate_i^{(j)}({\bf M})^{(j+1)} ={\bf M}^{(j+1)}.
\]
Whereas in the latter case 
\[
\mate_i^{(j)}({\bf M})^{(j-1)} ={\bf M}^{(j-1)},
\qquad
\mate_i^{(j)}({\bf M})^{(j+1)} =\vcenter{\hbox{\begin{tikzpicture}[scale=1.1, baseline=-.5ex]

\draw (0,0) rectangle (1.50,.84);
\node at (.76,.42) {\small{$e_i^{\mathrm{row}}(M_{j})$}};

\draw (0,.84) rectangle (1.50,1.68);
\node at (0.76,1.26) {\small{$M_{j+1}$}};

\draw[purple, -{Stealth[length=2mm,width=1.5mm]}, thick] (0,.84) -- (1.50,.84);
\draw[purple, -{Stealth[length=2mm,width=1.5mm]}, thick] (1.50,.84) -- (1.50,1.68);
\end{tikzpicture}}}
\, \, \text{or} \, \,
\vcenter{\hbox{\begin{tikzpicture}[scale=1.1, baseline=-.5ex]

\draw (0,0) rectangle (1.40,.84);
\node at (.72,.42) {\small{$e_i^{\mathrm{row}}(M_{j})$}};

{\phantomtile{0}{.84}{5}{3}{
}}

\draw (1.40,.84) rectangle (1.40+.84,1.68);
\node at (1.82,1.26) {\small{$M_{j+1}$}};

\draw[purple, -{Stealth[length=2mm,width=1.5mm]}, thick] (0,.84) -- (1.40,.84);
\draw[purple, -{Stealth[length=2mm,width=1.5mm]}, thick] (1.4,1.68) -- (1.4,.84);
\draw[purple, -{Stealth[length=2mm,width=1.5mm]}, thick] (1.4,1.68) -- (1.40+.84,1.68);
\end{tikzpicture}}}.\]

By \cite[Lemma~4.29]{AAA}, acting by a row crystal operator on a matrix only changes the column word up to Knuth equivalence, and similarly for a column operator and the row word, so 
$\operatorname{row}(e_i^{\mathrm{col}}(M_{j-1})) \sim \operatorname{row}(M_{j-1})$ and
$\overline{\operatorname{col}}(e_i^{\mathrm{row}}(M_j)) \sim \overline{\operatorname{col}}(M_j)$. 
By Lemmas~\ref{lem:duality-shift-equiv} and
\ref{lem:col-ins-row-ins}, together with the fact
that ordinary Knuth equivalence is preserved under concatenation,
concatenating the same barred or unbarred word preserves Knuth
equivalence when the resulting separated words have the same type.
Hence we see \eqref{eq:ei-matrix-equals-word} holds since:
 \begin{eqnarray}\nonumber
 {\bs v}^{(j-1)}  = &   \operatorname{row}(M_{j-2})\operatorname{row}(\worde_i^{\mathrm{col}}(M_{j-1})) &  \text{or} \;\; \overline{\operatorname{col}}(M_{j-2})\operatorname{row}(\worde_i^{\mathrm{col}}(M_{j-1}))\\ \nonumber
   \sim &  \operatorname{row}(M_{j-2})\operatorname{row}(M_{j-1})  & \text{or} \;\; \overline{\operatorname{col}}(M_{j-2})\operatorname{row}(M_{j-1})\\ \nonumber
    = &  {\bs w}^{(j-1)}; & \ \\ \nonumber
 {\bs v}^{(j+1)} = & \overline{{\operatorname{col}}}(e_i^{\mathrm{row}}(M_j))\overline{{\operatorname{col}}}(M_{j+1}) & \text{or} \;\; \overline{{\operatorname{col}}}(e_i^{\mathrm{row}}(M_j))\operatorname{row}(M_{j+1})\\ \nonumber
  \sim & \overline{{\operatorname{col}}}(M_j)\overline{{\operatorname{col}}}(M_{j+1}) & \text{or} \;\; \overline{{\operatorname{col}}}(M_j)\operatorname{row}(M_{j+1})\\
   = & {\bs w}^{(j+1)}. & \ \nonumber
 \end{eqnarray}

The statement about profiles follows from \eqref{eq:ei-matrix-equals-word} and 
Theorem~\ref{thm:stembridge-word-crystals} combined.
\end{proof}

\begin{lemma}\label{lem:single-vertex-injectivity}
Fix $j\in Q_0$, and let $\mathcal D$ be a connected component under
the $j$-crystal operators
\[\mate_i^{(j)},\matf_i^{(j)},\qquad 1\leq i<d_j.\]
Let ${\bs w}^{(j)}({\bf M})$ be the $j$-th entry of
${\sf read}_Q({\bf M})$. Then the map
\[{\bf M}\longmapsto {\bs w}^{(j)}({\bf M})\]
is injective on $\mathcal D$.
\end{lemma}
\begin{proof}
The $A_2$ cases $\left(
\begin{tikzpicture}[baseline=-0.5ex]
\node[circle,fill,inner sep=1pt,label=below:{\scriptsize $j$}] (a) at (0,0) {};
\node[circle,fill,inner sep=1pt,label=below:{\scriptsize $j+1$}]   (b) at (.9,0) {};
\draw[->] (a) -- node[above,font=\scriptsize] {$M_{j}$} (b);
\end{tikzpicture},
\begin{tikzpicture}[baseline=-0.5ex]
\node[circle,fill,inner sep=1pt,label=below:{\scriptsize $j$}] (a) at (0,0) {};
\node[circle,fill,inner sep=1pt,label=below:{\scriptsize $j+1$}]   (b) at (.9,0) {};
\draw[<-] (a) -- node[above,font=\scriptsize] {$M_{j}$} (b);
\end{tikzpicture}\right)$ are proved by \cite[Lemma~4.28]{AAA}. We need a consequence of that lemma, which is 
stronger than the $A_2$ case \emph{per se}:
In the former case, the 
$j$-crystal operators are the row bicrystal operators $e_i^{\operatorname{row}}, f_i^{\operatorname{row}}$, whereas in the
latter case they are the column bicrystal operators $e_i^{\operatorname{col}}, f_i^{\operatorname{col}}$. The
map from $d_j\times d_{j+1}$ nonnegative integer matrices $M$ to the word $\operatorname{row}(M)$
is injective on the set of matrices obtained from a fixed matrix by iterated application of the row bicrystal operators. Similarly, one has an 
injection on the map $M\to \overline{\operatorname{col}}(M)$ (up to the innocuous use of barred symbols) on the matrices obtained by using column bicrystal operators.

Now we claim the general case reduces to the $A_2$ analysis. We refer the reader to the four cases in the proof of Proposition~\ref{lem:matrix-quiver-crystal-ops-RSK} where the blocks occupied by $M_{j-1},M_j$ possibly affected by the $j$-crystal operators are depicted, along with their reading word. 

In \emph{Case 1} and \emph{Case 2}, 
the $j$-crystal operators are precisely the row and column bicrystal operators on $d_{j}\times (d_{j-1}+d_{j+1})$-size
and $(d_{j-1}+d_{j+1})\times d_{j}$-size matrices, respectively. Hence these cases immediately follow.

In \emph{Case 3}, the $j$-crystal operators are not exactly the
bicrystal operators on a single matrix, but each operator acts either
as a row bicrystal operator on $M_j$'s block or as a column bicrystal operator
on $M_{j-1}$'s block; this follows from Lemma~\ref{lem:separated-crystal-combined-word}. 
Thus, for any two elements of ${\mathcal D}$, deleting
from a path between them the steps acting on $M_{j-1}$'s block gives a
sequence of row bicrystal operators between the corresponding
$M_j$'s; similarly, deleting the steps acting on $M_j$'s block 
gives a sequence of column bicrystal operators between
the matrices occupying their respective $M_{j-1}$-blocks.
Here ${\bs w}^{(j)}({\bf M})$ consists of the barred segment
$\overline{\operatorname{col}}(M_{j-1})$ followed by the unbarred
segment $\operatorname{row}(M_j)$.
Since every $j$-crystal operator preserves the sum of the entries in
each block, the lengths of these two segments are fixed. Hence
${\bs w}^{(j)}({\bf M})$ determines the two segments separately, and
injectivity follows from two applications of the $A_2$ case. A similar argument applies to \emph{Case 4}.
\end{proof}

\begin{lemma}[Quiver crystal operator commutation]\label{lem:quiver-crystal-commute}
If $j\neq k$,
\[T_i^{(j)}\in\{\mate_i^{(j)},\matf_i^{(j)}\},
\qquad
U_r^{(k)}\in\{\mate_r^{(k)},\matf_r^{(k)}\},\]
then
\[T_i^{(j)}U_r^{(k)}({\bf M})=U_r^{(k)}T_i^{(j)}({\bf M}).\]
\end{lemma}
\begin{proof}
If $|j-k|\geq 2$, the two operators act on disjoint sets of matrix
blocks, so the claim is immediate. Thus assume $|j-k|=1$, and by symmetry assume
$k=j+1$. As in the proof of Lemma~\ref{lem:single-vertex-injectivity} we study the four cases
from Proposition~\ref{lem:matrix-quiver-crystal-ops-RSK}.

We will make use of this $A_2$ analysis. In \cite[Lemma~1.4.7]{van}
(restated in the notation below in \cite[Lemma~4.30]{AAA}): the row and column bicrystal operators commute, i.e.,
\[
\begin{aligned}
f_i^{\operatorname{row}}\!\left(f_r^{\operatorname{col}}(M)\right)
&=
f_r^{\operatorname{col}}\!\left(f_i^{\operatorname{row}}(M)\right), &
e_i^{\operatorname{row}}\!\left(e_r^{\operatorname{col}}(M)\right)
&=
e_r^{\operatorname{col}}\!\left(e_i^{\operatorname{row}}(M)\right),\\
e_i^{\operatorname{row}}\!\left(f_r^{\operatorname{col}}(M)\right)
&=
f_r^{\operatorname{col}}\!\left(e_i^{\operatorname{row}}(M)\right), &
f_i^{\operatorname{row}}\!\left(e_r^{\operatorname{col}}(M)\right)
&=
e_r^{\operatorname{col}}\!\left(f_i^{\operatorname{row}}(M)\right).
\end{aligned}
\]

Consider \emph{Case 1}. The $j$-crystal operator $T_i^{(j)}$ acts as a row bicrystal
operator on the $d_{j}\times (d_{j-1}+d_{j+1})$ block containing $M_{j-1}$ and $M_j$. The $(j+1)$-crystal 
operator $U_r^{(j+1)}$ either acts as a column bicrystal operator on $M_j$ or affects $M_{j+1}$. 

\begin{claim}\label{claim:July28zzz}
Suppose at least one of $U_r^{(j+1)}$ or $T_{i}^{(j)}$ when applied to ${\bf M}$ affects $M_j$. Then applying one
of these operators first does not affect whether the other operator returns $\o$, or which of its two adjacent blocks it
affects.
\end{claim}
\noindent\emph{Proof of Claim~\ref{claim:July28zzz}:} We make the argument for  \emph{Case 1}; a \emph{mutatis mutandis}
modification handles the other cases.  Suppose $U_{r}^{(j+1)}({\bf M})$ changes $M_j\mapsto M_j'$ by a column bicrystal operation. 
By \cite[Lemma~4.29]{AAA}, $\operatorname{row}(M_j)\sim \operatorname{row}(M_j')$. 
Now, by Proposition~\ref{prop:July24abc} applied to ${\mathcal K}_{d_j}$ one sees that the 
numbers of unmatched left and right parentheses
appearing in these Knuth equivalent two words are the same. From this, the claim 
follows by considering the definition of $T_i^{(j)}$ \emph{vis \`a vis} the words
${\operatorname{row}}(M_{j-1}) \operatorname{row}(M_j)$ 
and ${\operatorname{row}}(M_{j-1}) \operatorname{row}(M_j')$. 

The same argument with $T_i^{(j)}$ and $U_r^{(j+1)}$
interchanged proves the reverse statement, using the corresponding
barred-word version of the Knuth-equivalence argument from the proof
of Proposition~\ref{lem:matrix-quiver-crystal-ops-RSK}.
\qed

If $T_i^{(j)}$ affects $M_{j-1}$ and $U_{r}^{(j+1)}$ affects $M_{j+1}$ then clearly the operators commute.
The commutation proof in the other three possibilities of how the operators act follows from Claim~\ref{claim:July28zzz} and/or 
the $A_2$ analysis. This completes  \emph{Case 1}. \emph{Case 2} is similar.

Next suppose we are in \emph{Case 3}. As in the proof of \emph{Case 3} of 
Lemma~\ref{lem:single-vertex-injectivity}, the $j$-crystal operators $T_i^{(j)}$ either act as a row bicrystal
operator on $M_j$ or a column bicrystal operator on $M_{j-1}$. The $(j+1)$-crystal operators $U_{r}^{(j+1)}$
either act as column bicrystal operators on $M_j$ or act on $M_{j+1}$. As in \emph{Case 1}, the various subcases
of commutation are handled by Claim~\ref{claim:July28zzz} and/or the $A_2$ analysis. \emph{Case 4} is similarly reasoned.
\end{proof}

\begin{proof}[Proof of Theorem~\ref{prop:quiverRSK-crystal-isom}]
Let
\[
\Phi:{\mathcal C}\longrightarrow
{\mathcal K}_{\bf d}^{\operatorname{sep}},
\qquad
\Phi({\bf M})=[{\sf read}_Q({\bf M})].
\]
By Proposition~\ref{lem:matrix-quiver-crystal-ops-RSK}, $\Phi$ commutes with the
crystal operators, and it preserves weight by definition \eqref{eqn:Admwtdef}. We claim that $\Phi$ is
injective. Suppose ${\bf M},{\bf N}\in{\mathcal C}$ and
$\Phi({\bf M})=\Phi({\bf N})$.

Choose ${\bf M}_0\in{\mathcal C}$ and paths of quiver crystal
operators from ${\bf M}_0$ to ${\bf M}$ and from ${\bf M}_0$ to
${\bf N}$. By Lemma~\ref{lem:quiver-crystal-commute}, we may reorder
each path so that the $1$-crystal operators occur first, then all $2$-crystal operators, etc.

Let ${\bf M}^{(j)}$ and ${\bf N}^{(j)}$ be the matrices obtained
along the two reordered paths after the crystal operators for
vertices $1,\ldots,j$ have been applied. Thus
\[
{\bf M}^{(0)}={\bf N}^{(0)}={\bf M}_0,\qquad
{\bf M}^{(n)}={\bf M},\qquad
{\bf N}^{(n)}={\bf N}.
\]
We prove by induction on $j$ that
\begin{equation}\label{eqn:July24rst}
{\bf M}^{(j)}={\bf N}^{(j)}.
\end{equation}

The base case $j=0$ is by definition. For the inductive step, suppose
\[{\bf M}^{(j-1)}={\bf N}^{(j-1)}=: {\bf P}.\]
Both ${\bf M}^{(j)}$ and ${\bf N}^{(j)}$ are obtained from ${\bf P}$
using only $j$-crystal operators. 
By definition of the quiver
crystal operators, ${\bs w}^{(j)}({\bf M}^{(j)})$ and ${\bs w}^{(j)}({\bf N}^{(j)})$ are in the same connected component of the appropriate
$\mathcal W_{d_j}^{-+}$ or $\mathcal W_{d_j}^{+-}$.

Along the remainder of either reordered path, we only use $k$-crystal operators for
$k>j$. The $k$-crystal operators for $k\geq j+2$ keep the $j$-th reading word
fixed. Now, a $(j+1)$-crystal operator could alter the block occupied by $M_j$, 
and hence change the $j$-th reading word. However,
Proposition~\ref{lem:matrix-quiver-crystal-ops-RSK} shows that it
preserves its Knuth equivalence class. Combining with the assumption that
$\Phi({\bf M})=\Phi({\bf N})$ gives:
\[{\bs w}^{(j)}({\bf M}^{(j)})\sim {\bs w}^{(j)}({\bf M})\sim {\bs w}^{(j)}({\bf N})\sim {\bs w}^{(j)}({\bf N}^{(j)}).\]
By Theorem~\ref{thm:stembridge-word-crystals}, insertion is injective on
each connected component ${\mathcal E}$ of
${\mathcal W}_{d_j}^{\operatorname{sep}}$. Therefore a Knuth
equivalence class intersects ${\mathcal E}$ in at most one
word, so
\begin{equation}
\label{eqn:July27aaa}
{\bs w}^{(j)}({\bf M}^{(j)})={\bs w}^{(j)}({\bf N}^{(j)}).
\end{equation}
By construction and the inductive hypothesis, ${\bf M}^{(j)}$ and ${\bf N}^{(j)}$ are obtained from ${\bf P}$ using only the $j$-crystal operators.
Thus, in view of \eqref{eqn:July27aaa}, 
Lemma~\ref{lem:single-vertex-injectivity} implies \eqref{eqn:July24rst}. Thus ${\bf M}={\bf N}$ and
$\Phi$ is injective.

Let ${\mathcal D}$ be the connected component of
${\mathcal K}_{\bf d}^{\operatorname{sep}}$ containing
$\Phi({\bf M}_0)$. Since ${\mathcal C}$ is connected and $\Phi$
commutes with the crystal operators,
$\Phi({\mathcal C})\subseteq{\mathcal D}$.
Applying Proposition~\ref{prop:inj-implies-bij} now shows
$\Phi:{\mathcal C}\longrightarrow{\mathcal D}$
is a crystal isomorphism.

By Proposition~\ref{prop:July24abc}, insertion restricts to a
crystal isomorphism
\[
\operatorname{ins}:{\mathcal D}\longrightarrow{\mathcal B}_{\bf\Lambda}
\]
for some profile ${\bf\Lambda}$. We therefore have a composition of crystal isomorphisms:
\[
{\bf M}\longmapsto
[{\sf read}_Q({\bf M})]
\longmapsto
\operatorname{ins}({\sf read}_Q({\bf M}))
=
\operatorname{quiverRSK}({\bf M}).
\]
Hence $\operatorname{quiverRSK}|_{\mathcal C}:
{\mathcal C}\longrightarrow{\mathcal B}_{\bf\Lambda}$
is a crystal isomorphism.
\end{proof}

\begin{corollary} \label{cor:crystal-decomp-quiver-locus}
As $GL(\d)$--crystals,
\[\Std(\Omega) \cong \bigsqcup_{\mathcal{C}} \mathcal{B}_{{\bf \Lambda}(\mathcal{C})},\]
where the union is over the connected components of $\Std(\Omega)$ and ${\bf \Lambda}(\mathcal{C})$ is the profile of $\mathcal{C}$.
\end{corollary}

\begin{proof}
By Theorem~\ref{prop:quiverRSK-crystal-isom}, every connected component $\mathcal{C}$ of $\Std(\Omega)$ maps isomorphically onto the highest--weight crystal $\mathcal{B}_{{\bf \Lambda}(C)}$.
\end{proof}

We are now ready to state our crystal-theoretic version of Theorem~\ref{thm:newrule}.

\begin{theorem}\label{thm:mainprecise}
Let ${\bf \Lambda}$ be a sequence of rational shapes. Then
$m_{\Omega,{\bf \Lambda}}$ counts the number of ${\bf M}\in \Std(\Omega)$ such that
${\sf profile}({\bf M})={\bf \Lambda}$ and $\mate_i^{(j)}({\bf M})=\o$ for all $i,j$.

In view of Corollary~\ref{cor:superlowertab}, this is equivalent to counting those ${\bf M}\in \Std(\Omega)$ 
with $\sf{profile}({\bf M})=\Lambda$ such that if 
\[\operatorname{quiverRSK}({\bf M})=(U^{(1)}\div V^{(1)},\ldots,U^{(n)}\div V^{(n)})\]
then each $U^{(j)}$ is supersemistandard and each $V^{(j)}$ is lowest-weight.
\end{theorem}

\begin{corollary}\label{cor:nonzeroness}
If there exists ${\bf M}\in\Std(\Omega)$ with ${\sf profile}({\bf M})={\bf \Lambda}$ then
$m_{\Omega,{\bf \Lambda}}>0$.
\end{corollary}
\begin{proof}
Let ${\mathcal C}$ be the connected component containing ${\bf M}$.
By Corollary~\ref{cor:crystal-decomp-quiver-locus}, ${\mathcal C}\cong {\mathcal B}_{{\bf \Lambda}}$,
and hence ${\mathcal C}$ has a highest-weight element ${\bf M}'$ of
profile ${\bf \Lambda}$. By Theorem~\ref{thm:mainprecise}, $m_{\Omega,{\bf \Lambda}}>0$.
\end{proof}

The admissible matrices ${\bf M}$ satisfying Theorem~\ref{thm:mainprecise} are precisely those counted in
Theorem~\ref{thm:newrule}, once accounting for ordinification. The proof of both statements requires a \emph{monomial-positive} character
formula for ${\mathbb C}[\Omega]$, namely Theorem~\ref{thm:poschar} in the next section.

\section{Gr\"obner degeneration and the admissible matrix character formula}\label{sec:Hilbertseries}

This section proves a positive combinatorial rule for the character 
$\operatorname{char}_{T(\d)}({\mathbb C}[\Omega])$ of ${\mathbb  C}[\Omega]$ for the standard torus $T(\d)\leq GL(\d)$
consisting of invertible diagonal matrices
\[t=(t^{(i)})_{i\in Q_0}, \qquad t^{(i)}=\operatorname{diag}(t_1^{(i)},\ldots,t_{d_i}^{(i)}).\]
View ${\bf M}=(M_e)_{e\in Q_1}$ as an exponent matrix of a monomial in the polynomial ring
\[{\mathbb C}[\rep_Q(\d)]\cong {\mathbb C}[z_{cr}^{(e)}: e\in Q_1, 1\leq r\leq d_{h(e)}, 1\leq c\leq d_{t(e)}].\]
Here $z^{(e)}_{cr}$ denotes the variable occupying row $c$ and
column $r$ in the $Q$-shape; equivalently, it is the
coordinate function on the $(r,c)$-entry of the usual matrix in
$\Mat_{\d(h(e)),\d(t(e))}$. 
Each ${\bf M}$, as identified with the monomial 
\[z^{\bf M}:=\prod_{e\in Q_1}\prod_{c,r}
\left(z^{(e)}_{cr}\right)^{M_e(c,r)},\]
 is a weight vector for the action of $T(\d)$ on
${\mathbb C}[\rep_Q(\d)]$, 
the contragredient action to \eqref{eqn:bcgaction}. 

For each vertex $i\in Q_0=[1,n]$, introduce Laurent variables $\{{\sf z}_1^{(i)},\ldots,{\sf z}_{d_i}^{(i)}\}$.
This action 
assigns $z_{cr}^{(e)}$ the character ${\sf z}_c^{(t(e))}/{\sf z}^{(h(e))}_{r}$. Hence, ${\bf M}$ has weight
\[\operatorname{wt}_{T(\d)}({\bf M})=\prod_{e\in Q_1} 
\prod_{r=1}^{\d(h(e))} \prod_{c=1}^{\d(t(e))} \left({\sf z}_c^{(t(e))}/{\sf z}_r^{(h(e))}\right)^{M_e(c,r)}, \qquad
M_e=(M_e(c,r))_{1\leq c\leq \d(t(e)), 1\leq r\leq \d(h(e))}.\]
Equivalently, the exponent vector of
$\operatorname{wt}_{T({\bf d})}({\bf M})$ is
$\operatorname{wt}({\bf M})$ as defined in \eqref{eqn:Admwtdef}: an entry $M_e(c,r)$ contributes
$\varepsilon_c^{(t(e))}-\varepsilon_r^{(h(e))}$ to $\operatorname{wt}({\bf M})$.
 
\begin{theorem}
\label{thm:poschar}
\[\operatorname{char}_{T(\d)}({\mathbb C}[\Omega])=\sum_{{\bf M}\in \Std(\Omega)} \operatorname{wt}_{T(\d)}({\bf M}).\]
\end{theorem}

Theorem~\ref{thm:poschar} is the first monomial-positive formula for this character. Earlier work of Kinser--Knutson--Rajchgot \cite{KKR} uses the \emph{Kinser--Rajchgot reduction} \cite{KinserRajchgot} 
to the bipartite case to express the multigraded Hilbert series of $\mathbb C[\Omega]$, i.e., 
$\operatorname{char}_{T(\d)}({\mathbb C}[\Omega])$,
as a quotient, using a result of Woo and the third author \cite{WooYong} on \emph{Kazhdan-Lusztig ideals} \cite{WooYong:JA}. We derive Theorem~\ref{thm:poschar} using those results of \cite{KKR, WooYong}.

\subsection{Kazhdan-Lusztig ideals}\label{sec:KLideals}
Given $w\in S_N$, define the matrix $Z_w\in \mathrm{Mat}_{N,N}$ where:
\begin{itemize}
\item $1$'s are in matrix positions $(p,w(p))$ for $1\leq p\leq N$;
\item $0$'s are below and in the same column, or to the right of and in the same row, as a $1$;
\item $z_{pq}$ are in other positions $(p,q)$.
\end{itemize}
Let ${\mathbb C}[Z_w]$ be the polynomial ring in these $\ell(w)$ indeterminates $z_{pq}$.

Given $v\in S_N$, let $P_v$ be the permutation matrix with $1$'s in $(p,v(p))$ and $0$ elsewhere.
The \emph{rank function} 
$r_v:[N]^2\to {\mathbb Z}_{\geq 0}$ 
is defined by setting $r_v(p,q)$ to be the
number of $1$'s in $P_v$ in the northwest $p\times q$ region.
The \emph{Kazhdan-Lusztig ideal} \cite{WooYong:JA} is defined as
\begin{equation}\label{eqn:defininggens}
I_{w,v}\!=\!\langle\text{$r_v(p,q)\!+\!1$ size minors of the northwest $p\!\times\! q$ submatrix
of $Z_w$}\rangle.
\end{equation}
 It defines an irreducible affine variety ${\mathcal N}_{w,v}$, the \emph{Kazhdan-Lusztig variety}. We have $v\leq w$ in Bruhat order if and only if ${\mathcal N}_{w,v}\neq \emptyset$. 

Let $\prec_w$ be the antidiagonal lexicographic term order on the monomials in ${\mathbb C}[Z_w]$ defined by favoring variables
in columns further to the right and further north in a column. 
The \emph{initial ideal} $\operatorname{init}_{\prec_w}(I_{w,v})$ is generated by the $\prec_w$ 
leading terms of every $f\in I_{w,v}$. The 
$\prec$-\emph{standard basis} of 
${\mathbb C}[{\mathcal N}_{w,v}]\cong {\mathbb C}[Z_{w}]/I_{w,v}$ 
as a ${\mathbb C}$-vector space consists of monomials (taken to have coefficient $1$) not 
in $\operatorname{init}_{\prec_w}(I_{w,v})$, or equivalently, not divisible by the lead term of any generator in a Gr\"obner basis
of $I_{w,v}$ with respect to $\prec_w$.  We will need:

\begin{theorem}[\!\!\cite{WooYong}]\label{thm:GBresult}
The defining generators \eqref{eqn:defininggens}
of $I_{w,v}$ form a $\prec_w$-Gr\"obner basis. 
\end{theorem}

Since $\operatorname{init}_{\prec_w}(I_{w,v})$
is a monomial ideal, ${\mathbb C}[Z_{w}]/\operatorname{init}_{\prec_w}(I_{w,v})$ has a ${\mathbb Z}^{\ell(w)}$-multigrading
that assigns each $z_{pq}$ the weight ${\sf z}_{pq}$. We use this \emph{fine multigrading} in our derivation of Theorem~\ref{thm:poschar}. However, first, we need to study the bipartite lift $\widetilde \Omega$. 

\subsection{Relationship to $Q$-shapes of bipartite quivers}\label{subsec:theredux}
The work of \cite{KinserRajchgot} relates
${\bf M}$ as a  $Q$-shape to Kazhdan-Lusztig ideals. Assume the $A_n$
quiver $Q$ is \emph{bipartite}, i.e., 
there are $m+1$ source vertices labelled by the odd integers and $m$ sink vertices labeled by the even integers
alternating in $Q$; hence
$Q$ begins and ends with a source vertex and $n=2m+1$. The dimension vector is $\widetilde \d=(d_1,d_2,\ldots,d_{2m+1})$.
No replacement rules
are required to form ${\widetilde Q}=Q$. Place each of the variables $z_{cr}^{(i)}$ in the $\widetilde Q$-shape
(because of the transposed convention)
and complete this to a rectangular matrix $Z^{\widetilde Q}$ with $0$'s placed outside of the ${\widetilde Q}$-shape. This
$Z_{\widetilde Q}$ has $d_{\operatorname{odd}}:=d_1+d_3+\cdots+d_{2m+1}$ rows and $d_{\operatorname{even}}:=d_2+d_4+\cdots+d_{2m}$ columns. 

Let $\widetilde \Omega$ be a quiver locus in $\rep_{\widetilde Q}({\widetilde d})$, described by partial permutation
matrices $(L_1,\ldots,L_{2m})$.  Let $r_{[i,j]}$ be as in the introduction. 
Define an ideal 
$I_{\widetilde \Omega}\subset {\mathbb C}[{Z}_{\widetilde Q}]\cong {\mathbb C}[\rep_{\widetilde Q}({\widetilde \d})]$
generated by $r_{[i,j]}+1$ size minors of the smallest submatrix $Z^{\widetilde Q}_{[i,j]}$ of $Z^{\widetilde Q}$ containing ${\mathcal R}_{[i,j]}$.

By giving an equivariant isomorphism of ${\widetilde \Omega}$ with a Kazhdan-Lusztig variety ${\mathcal N}_{w,v}$ for certain
$w=w_{\widetilde \Omega}, v=v_{\widetilde\Omega} \in S_{N}$ depending on ${\widetilde \Omega}$ and $N=\sum_{i=1}^{2m+1} d_i$, Kinser--Rajchgot \cite{KinserRajchgot} prove:
\begin{theorem}[\!\!\!\cite{KinserRajchgot}]\label{thm:KRshortthm}
Let $\widetilde\Omega$ be a quiver locus in
$\rep_{\widetilde Q}(\widetilde{\d})$ for a bipartite quiver
$\widetilde Q$.
There is a $GL(\widetilde \d)$-equivariant isomorphism ${\mathbb C}[\widetilde \Omega] \cong {\mathbb C}[{Z}^{\widetilde Q}]/I_{\widetilde \Omega}$.
\end{theorem}

The proof of Theorem~\ref{thm:KRshortthm} from \cite{KinserRajchgot} (see also \cite{KKR}) shows there is a $GL(\widetilde \d)$-equivariant isomorphism ${\mathbb C}[\widetilde \Omega]\cong {\mathbb C}[Z_w]/I_{w,v}$ where $w=w_{\widetilde \Omega}$ is the dominant permutation
in $S_{d_1+\cdots +d_{2m+1}}$ whose Rothe diagram consists of a $d_{\operatorname{odd}}\times d_{\operatorname{even}}$ rectangle
to the northwest, and $v=v_{\tOmega}$ has a precise description omitted here. Thus, ${Z}_w$ is a copy of $Z^{\widetilde Q}$, with additional
variables outside of the ${\widetilde Q}$-shape, padded with $1$'s and $0$'s to the southeast.
We identify $z_{pq}=z_{cr}^{(i)}$ as appropriate inside the ${\widetilde Q}$-shape.

Let $\prec$ be the antidiagonal term order on the monomials in ${\mathbb C}[Z^{\widetilde Q}]$ that 
favors variables in columns further to the right and further north in a column of $Z^{\widetilde Q}$.
The following is a consequence of Theorem~\ref{thm:GBresult} and (the proof of) Theorem~\ref{thm:KRshortthm}.

\begin{proposition} \label{prop:degenstd}
The $\prec$-standard monomials of ${\mathbb C}[Z^{\widetilde Q}]/I_{\widetilde \Omega}$ are those 
monomials indivisible by any antidiagonal product of length $r_{[i,j]}+1$ from ${\mathcal R}_{[i,j]}$. 
\end{proposition}
\begin{proof}
By Theorem~\ref{thm:GBresult}, $I_{w_{\widetilde \Omega},v_{\widetilde \Omega}}$ has a $\prec_w$-Gr\"obner basis
given by its defining generators. By the construction of $v=v_{\tOmega}$, 
these generators consist of (1) the variables of ${\mathbb C}[Z_w]$ 
outside of the ${\widetilde Q}$-shape, and (2) the size $r_{[i,j]}+1$
determinants on the smallest submatrix of $Z_w$ that contains ${\mathcal R}_{[i,j]}$, and modulo the generators of type (1) these are minors of $Z^{\widetilde Q}_{[i,j]}$. 
In other words, there is an isomorphism
\[{\mathbb C}[Z_{w_{\widetilde \Omega}}]/I_{w_{\widetilde \Omega},v_{\widetilde \Omega}}\cong {\mathbb C}[Z^{\widetilde Q}]/I_{\widetilde \Omega}\] sending $z_{pq}$ to $z_{cr}^{(i)}$ occupying the corresponding position in $Z_{\widetilde Q}$. Since the term orders $\prec_w$ and $\prec$ agree 
on the monomials involving only ${\widetilde Q}$-shape, the result follows.
\end{proof}

\subsection{Fine grading and phantom coefficient extraction}

Let $I:=I_{\widetilde \Omega}\subseteq {\mathbb C}[Z^{\widetilde Q}]=:R$. We start working with its Gr\"obner degeneration 
$I':=\operatorname{init}_{\prec} I\subseteq R$. Consider the fine multigrading on $R$ that assigns
$z_{cr}^{(i)}$ the weight ${\sf z}_{cr}^{(i)}$. 
Since $I'$ is a monomial ideal, it is homogeneous for this grading.

Suppose ${\widetilde Q}$ is the modification of a  (possibly not bipartite) quiver $Q$ (if needed, add terminal source vertices of dimension $0$ with trivial maps without loss); $\tOmega$ is defined by the same partial permutation matrices as $\Omega$, together with the added identity matrices. 
Given ${\bf M}=(M_1,\ldots,M_{2m})$ thought of as a ${\widetilde Q}$-shape matrix, let
\[\operatorname{mwt}({\bf M})=\prod_{i=1}^{2m}\prod_{(c,r)\in {\mathcal R}_i}({\sf z}_{cr}^{(i)})^{M_i(c,r)}.\]

\begin{lemma}\label{prop:Jun18aaa}
The Hilbert series for $R/I'$ with respect to the fine multigrading is
\begin{equation}\label{eqn:Apr3-2}
\operatorname{Hilb}(R/I',{\sf z}) =\sum_{{\bf M}} \operatorname{mwt}({\bf M}), 
\end{equation}
where in \eqref{eqn:Apr3-2} the sum is over all ${\bf M}$ such that $z^{\bf M}$ is
a standard monomial, i.e., $z^M\not\in I'$.
\end{lemma}
\begin{proof}
Immediate from Proposition~\ref{prop:degenstd}.
\end{proof}

We now summarize additional results from \cite{WooYong} that we use.
Let $w,v\in S_N$ be arbitrary. In \cite[\S~4]{WooYong}, formulas are given for the $K$-polynomial and Hilbert series of ${\mathbb C}[Z_w]/I'_{w,v}$, where $I'_{w,v}=\operatorname{init}_{\prec_w}I_{w,v}$, with respect to the fine grading (in 
\emph{ibid.}~this corresponds to the rescaling action). This $K$-polynomial expression is the \emph{unspecialized Grothendieck polynomial} (\!\!\cite[Proposition~4.1]{WooYong}), itself described in terms of placements of $+$'s on the Rothe diagram of $w$ with the property that a certain \emph{Demazure product} associated to these placements gives
$v$ (modulo the $w_0$ arising from the upside-down
Rothe diagram convention in \cite{WooYong}). These placements, called
\emph{pipe dreams}, are denoted $\operatorname{Pipes}(w,v)$. We do not need to inspect
the definition of $\operatorname{Pipes}(w,v)$ in this paper except to say that the smallest size of any such pipe dream is $\ell(v)$. 
The $K$-polynomial formula is an alternating sum over $\operatorname{Pipes}(w,v)$.

The argument below restricts this discussion to the case $w=w_{\widetilde \Omega},v=v_{\widetilde \Omega}$. 
Above, we described $w_{\widetilde \Omega}$ and stated its Rothe diagram
is a $(d_{\operatorname{odd}}\times d_{\operatorname{even}})$-size rectangle $R$ where ${\widetilde Q}$-shape fits inside.
By the definition of $v_{\widetilde \Omega}$ from \cite{KKR}, if $\overline{\mathcal P}\in \operatorname{Pipes}(w_{\widetilde \Omega},v_{\widetilde 
\Omega})$ then there is a $+$
in every position of $R$ outside the ${\widetilde Q}$-shape. Now define ${\mathcal P}$ to be a \emph{pipe dream for 
$\widetilde \Omega$} if it agrees with some $\overline{\mathcal P}\in \operatorname{Pipes}(w_{\widetilde \Omega},v_{\widetilde \Omega})$ 
inside the ${\widetilde Q}$-shape. Let $\operatorname{Pipes}(\widetilde \Omega)$ be these pipe dreams
and $\ell(\widetilde\Omega)$ be the smallest number of $+$ appearing in any such object. 
By definition, 
adding back
the ``outside'' $+$'s gives a  bijection 
\begin{equation}
\label{eqn:Jul16yyy}
\operatorname{Pipes}(\widetilde \Omega) \xrightarrow{\sim}  \operatorname{Pipes}(w_{\widetilde \Omega},v_{\widetilde \Omega})
\end{equation}
 and 
$\ell(\widetilde\Omega)=\ell(v_{\widetilde\Omega})-q$ where $q$ is the number of ``outside'' $+$'s.

For ${\mathcal P}\in \operatorname{Pipes}(\widetilde \Omega)$, define
\[\operatorname{pwt}({\mathcal P})=\prod_{z_{cr}^{(i)}\in \mathcal P} (1-{\sf z}_{cr}^{(i)}).\]

\begin{lemma}[cf.~\cite{WooYong}]
\begin{equation}
\operatorname{Hilb}(R/I',{\sf z}) =  \frac{\sum_{\mathcal P} (-1)^{(\#+ \text{ in ${\mathcal P}$})-\ell({\tOmega})}
\operatorname{pwt}({\mathcal P})}{
\prod_{{\sf z}^{(i)}_{cr}\in Z_{\widetilde Q}}
\left(1-{\sf z}^{(i)}_{cr}\right)}  \label{eqn:Apr3-1}
\end{equation}
where in \eqref{eqn:Apr3-1} the sum is over all ${\mathcal P}\in\operatorname{Pipes}(\widetilde\Omega)$.
\end{lemma}
\begin{proof}
As explained in the proof of Proposition~\ref{prop:degenstd},
the projection that forgets the coordinates of
\(Z_{w_{\widetilde\Omega}}\) outside the \(\widetilde Q\)-shape
induces an isomorphism
$\mathbb C[Z_{w_{\widetilde\Omega}}]/
I_{w_{\widetilde\Omega},v_{\widetilde\Omega}}
\;\cong\;
\mathbb C[Z^{\widetilde Q}]/I_{\widetilde\Omega}$.
Moreover, the Gr\"obner-basis description used there shows that the
same projection induces an isomorphism of the initial quotients
$\mathbb C[Z_{w_{\widetilde\Omega}}]/
I'_{w_{\widetilde\Omega},v_{\widetilde\Omega}}
\;\cong\;
\mathbb C[Z^{\widetilde Q}]/I'
=
R/I'$.
Indeed, the variables outside the \(\widetilde Q\)-shape belong to the
initial ideal, while the remaining initial generators are precisely
those of \(I'\). Thus, by~\cite{WooYong} and the bijection~\eqref{eqn:Jul16yyy}
between
\(\operatorname{Pipes}(\widetilde\Omega)\) and
\(\operatorname{Pipes}(w_{\widetilde\Omega},
v_{\widetilde\Omega})\), the expression on the right-hand side of
\eqref{eqn:Apr3-1} equals
$\operatorname{Hilb}\!\left(
\mathbb C[Z_{w_{\widetilde\Omega}}]/
I'_{w_{\widetilde\Omega},v_{\widetilde\Omega}},
{\sf z}
\right)$ and thus also $\operatorname{Hilb}(R/I',{\sf z})$.
\end{proof}
 
Following \cite{KKR}, call ${\mathcal P}\in\operatorname{Pipes}(\widetilde\Omega)$ 
\emph{good} if it has no $+$ in a phantom region, and \emph{bad} otherwise.
We define the \emph{unspecialized Hilbert series for} ${\mathbb C}[\Omega]$ by
\begin{equation}\label{eqn:Apr3-3}
\overline{\operatorname{Hilb}}({\mathbb C}[\Omega],{\sf z}):=\frac{\sum_{\text{good ${\mathcal P}$}}
(-1)^{(\#+ \text{ in ${\mathcal P}$})-\ell({\tOmega})}
\operatorname{pwt}({\mathcal P})}{\prod_{\text{non-phantom ${\sf z}_{cr}^{(i)}$}} (1-{\sf z}_{cr}^{(i)})}. 
\end{equation}
As explained below, upon specialization of ${\sf z}_{cr}^{(e)}$,  $\overline{\operatorname{Hilb}}({\mathbb C}[\Omega],{\sf z})$ is precisely the Hilbert series formula for ${\mathbb C}[\Omega]$ from \cite{KKR}, but our unspecialized
version allows us to connect to the combinatorics of the bipartite case, in this calculation:

\begin{lemma}\label{lemma:Apr3aaa1}
$\overline{\operatorname{Hilb}}({\mathbb C}[\Omega],{\sf z})=\left[\prod_{{\text{\emph{phantom} $z^{(i)}_{cr}$}}} {\sf z}_{cr}^{(i)}\right] 
\operatorname{Hilb}(R/I',{\sf z})$,
where in the coefficient operator $\left[\prod_{{\text{\emph{phantom} $z^{(i)}_{cr}$}}} {\sf z}_{cr}^{(i)}\right]$, the product
is taken over all positions such that $z_{cr}^{(i)}$ is in a phantom region. 
\end{lemma}
\begin{proof}
Apply the coefficient operator
$\left[\prod_{{\text{phantom $z_{cr}^{(i)}$}}}
{\sf z}_{cr}^{(i)}\right]$
to \eqref{eqn:Apr3-1}. If $\mathcal P$ is good, then
$\operatorname{pwt}(\mathcal P)$ is independent of the phantom variables, and hence
\[
\left[\prod_{\text{phantom $z_{cr}^{(i)}$}}
{\sf z}_{cr}^{(i)}\right]
\frac{\operatorname{pwt}(\mathcal P)}
{\prod_{i,r,c}(1-{\sf z}_{cr}^{(i)})}
=
\frac{\operatorname{pwt}(\mathcal P)}
{\prod_{\text{non-phantom }z_{cr}^{(i)}}(1-{\sf z}_{cr}^{(i)})},
\]
since the coefficient of
$\prod_{\text{phantom $z_{cr}^{(i)}$}}{\sf z}_{cr}^{(i)}$
in the product of geometric series is $1$.

If $\mathcal P$ is bad, then for some phantom position $(i,r,c)$,
the factor $(1-{\sf z}_{cr}^{(i)})$ divides
$\operatorname{pwt}(\mathcal P)$. Cancelling this factor with the
denominator leaves an expression independent of ${\sf z}_{cr}^{(i)}$,
so its coefficient of
$\prod_{\text{\emph{phantom} $z_{cr}^{(i)}$}}{\sf z}_{cr}^{(i)}$
is $0$.

Thus only the good pipe dreams contribute, giving
$\overline{\operatorname{Hilb}}(\mathbb C[\Omega],{\sf z})$.
\end{proof}

\noindent\emph{Proof of Theorem~\ref{thm:poschar}:} In view of Lemma~\ref{lemma:Apr3aaa1}
and \eqref{eqn:Apr3-2},
\[
\overline{\operatorname{Hilb}}({\mathbb C}[\Omega],{\sf z})=\left[\prod_{{\text{phantom $z^{(i)}_{cr}$}}} {\sf z}_{cr}^{(i)}\right]
\sum_{{\bf M}} \operatorname{mwt}({\bf M}),
\]
where the sum is over ${\bf M}$ such that $z^{\bf M}\not\in I'$. If ${\bf M}$ contributes to the coefficient,
 $M_i(c,r)=1$
for every phantom $Q$-shape position $(c,r)$ in region ${\mathcal R}_i$. Define ${\bf A}=(A_{1},\ldots,A_{2m})$ to agree with ${\bf M}$ except at those phantom positions,
and where $A_{i}(c,r)=0$ at a phantom position. Clearly,
\[\operatorname{mwt}({\bf A})=\frac{\operatorname{mwt}({\bf M})}{\prod_{{\text{phantom $z_{cr}^{(i)}$}}} {\sf z}_{cr}^{(i)}}.\]
Moreover, after omitting the phantom regions, the map ${\bf M}\mapsto {\bf A}$ is a bijection between standard monomials of $I'$ having
exponent~$1$ at every phantom position and $\operatorname{Adm}(\Omega)$.
Indeed, by Proposition~\ref{prop:degenstd},
$z^{\bf M}\notin I'$ precisely when, for every region \(R_{[i,j]}\),
the positive support of ${\bf M}$ contains no antidiagonal of length
$r_{[i,j]}+1$. Since $M_i(c,r)=1$ at every phantom position,
after those positions are erased this is precisely stating that
${\bf A}$ contains no (phantom) antidiagonal of length $r_{[i,j]}+1$.

Thus, as an equality of generating series:
\begin{equation}\label{eqn:Apr3-5}
\overline{\operatorname{Hilb}}({\mathbb C}[\Omega],{\sf z})=\sum_{{\bf A}\in \Std(\Omega)} \operatorname{mwt}({\bf A}).
\end{equation}

Finally, we invoke \cite{KKR}: Each $e\in Q_1$ is associated with an edge $i(e)\in [2m]$ of ${\widetilde Q}_1$.  
Upon the substitution 
\[{\sf z}_{cr}^{(i(e))}\mapsto {\sf z}_c^{(t(e))}/{\sf z}_r^{(h(e))}, e\in Q_1,\]
$\overline{\operatorname{Hilb}}({\mathbb C}[\Omega],{\sf z})$ equals the Hilbert series of ${\mathbb C}[\Omega]$
for the $T(\d)$-grading, i.e., $\operatorname{char}_{T(\d)}{\mathbb C}[\Omega]$.
 Hence the same substitution on the righthand side of \eqref{eqn:Apr3-5}
gives the desired expression for the same Hilbert series (i.e., character), as a weighted sum over $\operatorname{Adm}(\Omega)$.\qed

\section{Character theory and the proof of Theorems~\ref{thm:newrule} $\&$~\ref{thm:mainprecise}}\label{sec:proof-cont}
The rational character ring of $GL_d$ is 
\[R(GL_d)=\mathbb Z[x_1^{\pm1},\ldots,x_d^{\pm1}]^{{\mathfrak S}_d},\] 
where
${\mathfrak S}_d$ is the symmetric group on $[d]$. It has a ${\mathbb Z}$-basis of  the irreducible characters
$\operatorname{char}_{T(d)}(V_{\alpha}(d))$ where $V_{\alpha}(d)$ is the irreducible representation of $GL_d$ of highest
weight $\alpha=(\alpha_1\geq \alpha_2\geq\cdots\geq \alpha_{d})\in {\mathbb Z}^d$; here $T(d)\subset GL_d$ is the standard maximal torus. If $\alpha=\lambda-\Delta\cdot(1,1,\ldots,1)$ then 
\[\operatorname{char}_{T(d)}(V_{\alpha}(d))=(x_1\cdots x_d)^{-\Delta}s_{\lambda}(x_1,\ldots,x_d),\] 
where $s_{\lambda}(x_1,\ldots,x_d)$ is the \emph{Schur polynomial}. Recall, 
\[V_{[\lambda,\Delta]}(d)\cong V_{[\widetilde\lambda,\widetilde\Delta]}(d) \text{ if $\lambda-\Delta\cdot(1,1,\ldots,1)=\widetilde\lambda-\widetilde\Delta\cdot(1,1,\ldots,1)$.}\] 
Furthermore, 
$R(GL(\d)) \cong \bigotimes_{i=1}^n R(GL_{d_i})$.
Consequently, $R(GL(\d))$ has a ${\mathbb Z}$-basis of characters $\operatorname{char}_{T(\d)}(V_{{\bf \Lambda}})$ where ${\bf \Lambda}$ is a sequence of highest weights (whether described by partitions, rational shapes or a sequence of pairs $[\lambda,\Delta]$). See \cite{Stembridge} for more details.

The coordinate ring $\mathbb{C}[\Omega]$ has a unique isotypic decomposition 
$\mathbb{C}[\Omega] \cong \bigoplus_{{\bf \Lambda}}  V_{{\bf \Lambda}}^{\oplus m_{\Omega,{\bf \Lambda}}}$.
All formal-character identities below are understood degree-wise; i.e.,
they lie in $R(GL(\d))[[t]]$, with the variable $t$ suppressed from the notation.
\begin{equation}
\label{eqn:June24rrr}
\operatorname{char}_{T(\d)}({\mathbb C}[\Omega])=
\sum_{{\bf \Lambda}} m_{\Omega, {\bf \Lambda}}\, \text{char}_{T(\d)}(V_{{\bf \Lambda}}).
\end{equation}

For a weight $\mu^{(i)}=(\mu_1^{(i)},\ldots,\mu_{d_i}^{(i)})\in {\mathbb Z}^{d_i}$ of the torus $T(d_i)$, define 
\[{\sf z}^{\mu^{(i)}}=\prod_{\ell=1}^{d_i} ({\sf z}_{\ell}^{(i)})^{\mu_{\ell}^{(i)}}.\]
For a weight of the product of tori $T(\d)$, $\mu=(\mu^{(1)},\ldots,\mu^{(n)})\in {\mathbb Z}^{d_1}\times\cdots
\times {\mathbb Z}^{d_n}$, define
\[{\sf z}^{\mu}=\prod_{i\in Q_0}\prod_{\ell=1}^{d_i} ({\sf z}_{\ell}^{(i)})^{{\mu}_{\ell}^{(i)}}.\]

Now for an ordinary partition $\gamma$ with at most $d_i$ parts, define
\[\operatorname{char}({\mathcal B}_{\gamma}^{(d_i)})
=\sum_{T\in {\mathcal B}_{\gamma}^{(d_i)}} {\sf z}^{\operatorname{wt}_{{\mathcal B_{\gamma}}}(T)}.\]
This expression equals 
$s_\gamma\bigl({\sf z}^{(i)}_1,\ldots,{\sf z}^{(i)}_{d_i}\bigr) =\operatorname{char}_{T(d_i)}(V_\gamma(d_i))$.
More generally, if $\gamma$ is a rational shape, by Proposition~\ref{prop:tableau-shift-isomorphism} it follows that
$\operatorname{char}({\mathcal B}_{\gamma}^{(d_i)})=\operatorname{char}_{T(d_i)}(V_\gamma(d_i))$.
Lastly, for ${\bf{\Lambda}}=(\gamma_1,\ldots,\gamma_n)$ and ${\mathcal B}_{{\bf \Lambda}}=\prod_{i\in Q_0}
{\mathcal B}_{\gamma_i}^{(d_i)}$ (Section~\ref{subsec:Cryprofiles}) define
\begin{equation}
\label{eqn:July18ccc}
\operatorname{char}({\mathcal B}_{\bf \Lambda})=\sum_{{\bf T}=(T_1,\ldots,T_n)\in {\mathcal B}_{\bf \Lambda}}
{\sf z}^{\operatorname{wt}_{{\mathcal B}_{\bf\Lambda}}({\bf T})}.
\end{equation}
Since the weight function ${wt}_{{\mathcal B}_{\bf\Lambda}}$ is defined componentwise by \eqref{eqn:July18aaa},
\begin{equation}
\label{eqn:July18bbb}
\operatorname{char}({\mathcal B}_{\bf{\Lambda}})=\prod_{i\in Q_0} \operatorname{char}({\mathcal B}_{\gamma_i}^{(d_i)})
=\operatorname{char}_{T(\d)}(V_{\bf\Lambda}).
\end{equation}
Since insertion preserves weight, if ${\bf M}$ has profile
${\bf\Lambda}$, then
\[
\operatorname{wt}({\bf M})
=
\operatorname{wt}_{{\mathcal B}_{\bf\Lambda}}
\bigl(\operatorname{quiverRSK}({\bf M})\bigr).
\]
Therefore,
\[\operatorname{wt}_{T({\bf d})}({\bf M})
= {\sf z}^{\operatorname{wt}({\bf M})}
= {\sf z}^{\operatorname{wt}_{{\mathcal B}_{\bf\Lambda}}
(\operatorname{quiverRSK}({\bf M}))}.\]
Hence,
\[\operatorname{char}_{T(\d)}(\mathbb{C}[\Omega])\stackrel{\text{Thm.~\ref{thm:poschar}}}{=}
\sum_{{\bf M}\in \Std(\Omega)} \!\!\wt_{T(\d)}({\bf M}) 
\stackrel{\text{Cor.~\ref{cor:crystal-decomp-quiver-locus}}}{=}
\sum_{\mathcal{C}}\sum_{{\bf T}\in \mathcal{B}_{{\bf \Lambda}(\mathcal{C})} } {\sf z}^{\wt({\bf T})} 
\stackrel{\eqref{eqn:July18ccc}, \eqref{eqn:July18bbb}}{=} \sum_{\mathcal{C}} \text{char}_{T(\d)}(V_{{\bf \Lambda}(\mathcal{C})}).\]
Comparing with \eqref{eqn:June24rrr}, and using the linear independence of irreducible characters, we see $m_{\Omega,{\bf \Lambda}}$ equals the number of connected components $\mathcal{C}$ of profile ${\bf \Lambda}$. 
By Corollary~\ref{cor:crystal-decomp-quiver-locus}, each connected component ${\mathcal C}$ is
isomorphic to ${\mathcal B}_{{\bf \Lambda}({\mathcal C})}$, and hence has a
unique highest-weight element. Therefore, $m_{\Omega,{\bf \Lambda}}$ counts
highest-weight elements ${\bf M}\in\Std(\Omega)$ with
$\operatorname{profile}({\bf M})={\bf \Lambda}$, proving Theorem~\ref{thm:mainprecise}.

By Proposition~\ref{prop:ordinification-isomorphism}, ordinification preserves the
highest-weight condition. Thus the highest-weight matrices ${\bf M}$ appearing in
Theorem~\ref{thm:mainprecise} are precisely the highest-weight matrices of
Theorem~\ref{thm:newrule}. Moreover, the number of barred letters in the $i$-th entry of 
${\sf read}_Q({\bf M})$  is $\Delta_i$, so
Lemmas~\ref{lem:duality-shift-equiv} and~\ref{lem:col-ins-row-ins}
identify the rational shape at vertex $i$ with the highest-weight indicated by $[\lambda^{(i)},\Delta_i]$. Theorem~\ref{thm:newrule} follows.
\qed

\section{Application to the variety of complexes case}\label{sec:apptocomplex}

We now apply our main result, Theorem~\ref{thm:newrule} (in its rephrased form, Theorem~\ref{thm:mainprecise}), to derive
a simple formula in the variety of complexes case. That is, we assume that the quiver locus corresponds
to the situation
\begin{equation}
\label{eqn:thesituation}
{\mathbb C}^{d_1} \xrightarrow{L_1} {\mathbb C}^{d_2} \xrightarrow{L_2} \cdots \xrightarrow{L_{n-1}} 
{\mathbb C}^{d_n} \text{\ such that } L_{i+1}L_i={\bf 0}, 1\leq i\leq n-2.
\end{equation}
Thus, the dimension vector of our equioriented $A_n$-quiver $Q$ is $\d=(d_1,d_2,\ldots,d_n)$.

Fix a collection of nonnegative 
integers ${\bf r}=(r_i)_{i=0}^{n}$ where $r_0=r_{n}=0$ and $r_i+r_{i+1}\leq d_{i+1}$ for $0\leq i\leq n-1$.
Any nonempty quiver locus $\Omega$ of \eqref{eqn:thesituation} is described as
\[\Omega=\Omega_{\bf r}=\{(L_1,\ldots,L_{n-1}): \operatorname{rank}(L_i)\leq r_i, L_{i+1}L_i={\bf 0}\}\]
for such ${\bf r}$; see, e.g., the rank array viewpoint for equioriented quivers in \cite{BuchFulton}. Given partitions
$\lambda_1,\lambda_2,\ldots,\lambda_{n-1}$ such that $\ell(\lambda_i)\!\leq \! r_i$ for $1\!\leq i\leq\! n-1$, let
\[{\bf \Lambda}(\lambda_1,\ldots,\lambda_{n-1})=(\lambda_1\div \varnothing,\lambda_2\!\div\!\lambda_1,\lambda_3\!\div\!\lambda_2,\ldots,
\lambda_{n-1}\!\div\!\lambda_{n-2},\varnothing\!\div\!\lambda_{n-1}).\]

De Concini and Strickland~\cite[Theorem~1.7]{Strickland} give a basis for the coordinate ring using \emph{standard multitableaux}. This basis can then be used to determine the $GL$-decomposition of the coordinate ring and its multiplicity-freeness (see~\cite[Remark p.~70]{Strickland}). However, the rule below does not appear explicitly in \cite{Strickland}, and in any case is derived using different methods.

\begin{corollary}[{cf.~\cite[Remark p.~70]{Strickland}}]\label{thm:complexesapp}
If $\Omega_{\bf r}\subset \rep_Q(\d)$ then $\mathbb{C}[\Omega_{\bf r}]$ is multiplicity-free. Indeed,
$m_{\Omega_{\bf r},{\bf \Lambda}}=1$ if ${\bf \Lambda}={\bf \Lambda}(\lambda_1,\ldots,\lambda_{n-1})$
and $m_{\Omega_{\bf r},{\bf \Lambda}}=0$ otherwise.
\end{corollary}
\begin{proof} Let 
\[V_i=A_i\oplus B_i\oplus C_i, \qquad \dim A_i=r_{i-1}, \qquad \dim B_i = r_i, \qquad \dim C_i=d_i-r_{i-1}-r_i.\] 
Define $L_i$ to be an isomorphism 
$B_i\xrightarrow{\sim} A_{i+1}$ and
zero on $A_i\oplus C_i$. Hence $\operatorname{rank} \, L_i=r_i$ and $\operatorname{im} \, L_i
=A_{i+1}\subseteq \operatorname{ker} \, L_{i+1}$ and therefore $L_{i+1}L_i={\bf 0}$. 
Picking bases, the $Q$-shape alternates between 
partial permutation matrices $L_i$ in ${\mathcal R}_{2i-1}$ and phantom blocks in ${\mathcal R}_{2i}$. 
Inspecting, for $i<j$, the rank of ${\mathcal R}_{[2i-1,2j-1]}$ is 
$d_{i+1}+d_{i+2}+\cdots+d_{j}$. It follows (by pigeonhole) that
${\bf M}\in \Std(\Omega_{\bf r})$ if and only if ${\mathcal R}_{2i-1}$ avoids
an antidiagonal of length $r_i+1$  for $1\leq i\leq n-1$, and each three-block region ${\mathcal R}_{[2i-1,2i+1]}$ avoids 
a (phantom) antidiagonal of length $d_{i+1}+1$ for $1\leq i\leq n-2$.

We first prove the claim for $n=3$ where 
${\bf M} =\vcenter{\hbox{\begin{tikzpicture}[scale=.9, baseline=-.5ex]
\draw (0,0) rectangle (.84,.84);
\node at (.42,.42) {\small{$M_1$}};

{\phantomtile{0}{.84}{3}{3}{
}}

\draw (.84,.84) rectangle (1.68,1.68);
\node at (1.26,1.26) {\small{$M_2$}};

\draw[purple, -{Stealth[length=2mm,width=1.5mm]}, thick] (0,.84) -- (.84,.84);
\draw[purple, -{Stealth[length=2mm,width=1.5mm]}, thick] (.84,1.68) -- (.84,.84);
\draw[purple, -{Stealth[length=2mm,width=1.5mm]}, thick] (.84,1.68) -- (1.68,1.68);
\end{tikzpicture}}}$.
Suppose ${\bf M}$ contributes to $m_{\Omega,\Lambda}$ under the rule of Theorem~\ref{thm:mainprecise}.
Now,
\[{\sf read}_Q({\bf M}) = (\operatorname{row}(M_1), \, \overline{\operatorname{col}}(M_1)\operatorname{row}(M_2), \,
\overline{\operatorname{col}}(M_2))=: ({\bs w}_1, {\bs w}_2, {\bs w}_3).\] 
Hence 
\[{\sf profile}({\bf M})= (\lambda_1, \mu_2\div \mu_1, \varnothing\div \lambda_2) \text{ for some partitions $\lambda_1,\mu_1,\mu_2,\lambda_2$.}\]

Note that $\operatorname{row}(M_2)$ is a highest-weight word since it is the suffix of the highest weight word
${\bs w}_2^{+}$. By Greene's theorem on RSK, $M_2$ has a longest antidiagonal of length
$j=\ell(\lambda_2)\leq~r_2$ using rows $1,2,\ldots,j$. Such an 
antidiagonal extends to one
of length $d_2$ by picking positions in the phantom region such that the southeastmost choice is in row $d_2$
and column $j+1$ of the phantom region. Since ${\bf M}$ does not contain a length $d_2+1$ antidiagonal, $M_1$ has no positive entries in columns $1,2,\ldots,j$. 

Let $\ins_r(\overline{\operatorname{col}}(M_1))=\varnothing\div \widetilde{V}$ of shape $\varnothing\div \widetilde{\mu_1}$. Since 
$\ins_r(\operatorname{row}(M_1))$ and $\ins_r(\operatorname{col}(M_1))$ are the same shape
(being RSK insertion and recording tableaux of $M_1$), $\widetilde{\mu_1}=\lambda_1$. Suppose 
\[U\div V:=\ins_r({\bs w}_2)=(\varnothing \div \widetilde{V} \leftarrow \operatorname{row}(M_2)).\] 
By definition, 
$U\div V=\operatorname{shift}^{-k}(\operatorname{shift}^k(\varnothing\div \widetilde{V})\leftarrow \operatorname{row}(M_2)), k\gg 0$. 
Since none of $\barindex{1},\barindex{2},\ldots,\barindex{j}$ appear in $\overline{\operatorname{col}}(M_1)$, 
the first $k$ columns of $\operatorname{shift}^k(\varnothing\div \widetilde{V})$ each use $1,2,\ldots,j$. 
Since $\operatorname{row}(M_2)$ only uses these
labels, its insertion into $\operatorname{shift}^k(\varnothing\div \widetilde{V})$ does not affect these columns, so applying $\operatorname{shift}^{-k}$, we 
see $V=\widetilde{V}$ and $U=\ins_r(\operatorname{row}(M_2))$. Hence $\mu_1=\lambda_1$ and $\mu_2=\lambda_2$. 
Thus, $m_{\Omega,{\bf \Lambda}}=0$ unless ${\bf \Lambda}$ has the desired form. 

Now,
$m_{\Omega,{\bf \Lambda}}\leq 1$ since if ${\bf M}\in\Std(\Omega)$ has
${\sf profile}({\bf M})=(\lambda_1,\lambda_2\div\lambda_1,\varnothing\div\lambda_2)$ and is 
highest-weight, then $\operatorname{row}(M_2)$ and $\overline{\operatorname{col}}(M_2)$ must be highest-weight (that is, $\operatorname{col}(M_2)$ is lowest-weight), and since $\operatorname{row}(M_2)$ and $\overline{\operatorname{col}}(M_1)$ have disjoint underlying labels, $\operatorname{col}(M_1)$ must be lowest-weight as well. This uniquely determines $M_1$ and $M_2$, by ordinary RSK.  
Conversely, choose $M_i$ of RSK shape $\lambda_i$, for $i=1,2$,
so that $\operatorname{row}(M_i)$ and
$\overline{\operatorname{col}}(M_i)$ are highest-weight.
Then $\operatorname{row}(M_2)$ uses only
$1,\ldots,\ell(\lambda_2)$, while $\operatorname{col}(M_1)$ uses only
the last $\ell(\lambda_1)$ labels. Since
$\ell(\lambda_1)+\ell(\lambda_2)\leq d_2$,
these sets of labels are disjoint. Hence
${\bf M}\in\Std(\Omega)$, and the shift-insertion argument above gives
${\sf profile}({\bf M})
=(\lambda_1,\lambda_2\div\lambda_1,\varnothing\div\lambda_2)$.
By Corollary~\ref{cor:nonzeroness}, $m_{\Omega,{\bf\Lambda}}\geq~1$.

For general $n$, the same argument applies to each middle word
$\overline{\operatorname{col}}(M_i)\operatorname{row}(M_{i+1})$.
For existence, choose each $M_i$ so that its row word is highest-weight
and its column word is lowest-weight. Then
$\operatorname{row}(M_{i+1})$ uses only the first
$\ell(\lambda_{i+1})$ labels, while $\operatorname{col}(M_i)$ uses only
the last $\ell(\lambda_i)$ labels. Since
$\ell(\lambda_i)+\ell(\lambda_{i+1})\le d_{i+1}$,
these sets of labels are disjoint for every internal vertex.
\end{proof}

\begin{remark}[Spherical $\rep_Q({\bf d})$]
A.~Knutson (private communication) raised the question of providing an RSK bijection where $\rep_Q({\bf d})$ is spherical with respect
to $GL(\d)$; see the recent classification \cite{MullerReineke}. 
The space $\rep_Q(\mathbf d)$ is spherical if and only if
$\mathbb C[\rep_Q(\mathbf d)]$ is multiplicity-free as a $GL(\mathbf d)$-representation. Using Theorem~\ref{thm:newrule}, we can see that this happens if and only if $\d$ avoids:
\begin{itemize}
\item contiguous subsequences $(a,b,c)$ where $a\geq 2, b\geq 3, c\geq 2$;
\item contiguous subsequences $(a,b,c,d)$ where $a\geq 1, b\geq 2, c\geq 2, d\geq 1$.
\end{itemize}
$\operatorname{quiverRSK}$ provides the desired bijection
between $\Std(\rep_Q(\d))$ and sequences of rational tableaux.
\end{remark}

%

\subsection*{Acknowledgements}
This work is related to earlier joint work of the third author with Abigail Price and Ada Stelzer on bicrystalline ideals \cite{AAA,AAA2}, and we have benefited from many discussions with them. 
We thank Zoe Lerdworatawee for her work in ICLUE, which computer-checked some of our ideas. We also thank Tim Chow, Steven Sam, and Edward Tong for helpful discussions. 
We used generative-AI tools ({\tt ChatGPT Plus} and {\tt Claude}) in the final stage to help 
check our arguments, proofread, and typeset.  The authors take full responsibility for all content.
This work was partially supported by NSF RTG in Combinatorics DMS 1937241 and a Simons Collaboration Grant to AY.

\bibliographystyle{siam}
\bibliography{bibliography.bib}

\end{document}